\documentclass[12pt]{amsart}
\usepackage{amsmath, amssymb, latexsym, amsthm,bm}
\usepackage{url}
\usepackage{mathtools}
\usepackage{mathrsfs}
\usepackage{color}
\usepackage[rgb]{xcolor}
\usepackage{tikz}
\usepackage{graphicx}
\usepackage{stmaryrd,amsbsy}
\usepackage{mathabx}
\usepackage{enumitem}
\usepackage{multirow}
\usepackage{textcomp}
\usepackage[sort,nocompress]{cite}
\usepackage{mathscinet}
\usetikzlibrary{calc}
\usepackage{cancel}
\usepackage{xspace}
\usepackage[mathlines]{lineno}
\usepackage{url}
\usepackage{longtable}
\usepackage[obeyFinal,colorinlistoftodos,textsize=tiny]{todonotes}
\usepackage[font=small]{caption}
\usepackage{array}
\usepackage{tikz}
\usepackage{stmaryrd}
\usepackage[author={Max Schlepzig}]{pdfcomment}
\usepackage{cleveref}
\usepackage{comment}
\usepackage{relsize}
\usepackage{soul}
\setlist{  
  listparindent=12pt,
  parsep=0pt,
}

\def\linksor#1{{\overrightarrow{\LL}^+(#1)}}
\def\linksu#1{{{\LL}^+(#1)}}
\def\aa{{\mathscr A}}
\def\numb#1{{\llbracket{\linksor{#1}}\rrbracket}}
\def\numbor#1{{\llbracket{\linksor{#1}}\rrbracket}}
\def\numbu#1{{\llbracket{\linksu{#1}}\rrbracket}}
\def\LL{{\mathcal L}}
\newcommand*{\Scale}[2][4]{\scalebox{#1}{\ensuremath{#2}}}%

\def\rr{\mathscr R}
\def\bb{{\mathscr B}}
\def\ff{{\mathscr F}}
\def\Isot#1#2#3{{ {#1}\, \bigl| \, {#2}\to{#3}}}
\def\ii{{\mathcal I}}
\def\isot#1#2{{{#1}\mapsto{#2}}}
\def\identity{\mbox{\Large$\iota$}}
\def\reverse{{\nu}}
\def\isop#1#2#3{{{#1}\stackrel{#3}{\longmapsto}{#2}}}
\def\Isop#1#2#3#4{{ {#1} \,\bigl| \,  {#2}\stackrel{#4}{\longrightarrow}{#3}}}
\def\jj{{\mathcal J}}
\def\In#1{{{#1}^{-1}}}
\def\Ov#1{{\overline{#1}}}
\def\IO#1{{\In{(\Ov{#1})}}}
\def\Ho{{\mathcal H}}
\def\Ve{{\mathcal V}}
\def\ZZ{{\mathbb Z}}
\definecolor{lgray}{rgb}{0.85, 0.85, 0.85}
\def\gl#1{{\textcolor{lgray}{#1}}}
\def\oL{{{L^*}}}
\def\kk{{\mathcal K}}
\def\Ro{{\mathcal R}}
\def\ro{{\mathbf r}}
\def\rot#1{{\sigma}}
\def\opi#1{{\,\oplus_{{}_{#1}}}}
\def\ve{{\mathbf v}}
\def\grou#1#2{{\langle{#1},{#2}\rangle}}
\def\age{{\mathbf g}}

\newtheorem{theorem}{Theorem} 
\newtheorem{theorem*}{Theorem} 
\newtheorem{proposition}[theorem]{Proposition} 
\newtheorem{fact}[theorem]{Fact} 

\newtheorem{corollary}[theorem]{Corollary}
\newtheorem{lemma}[theorem]{Lemma}

\newtheorem{observation}[theorem]{Observation}

\newtheorem{claim}[theorem]{Claim}
\newtheorem{remark}[theorem]{Remark}
\newtheorem{question}[theorem]{Question}

\theoremstyle{definition}

\MakeRobust{\overrightarrow}

\def\rank#1{{\text{\rm rank}}(#1)}

\begin{document}




\title[{Positive links with arrangements of pseudocircles as shadows}]{{Positive links with arrangements of pseudocircles \\ as shadows}}

\author[Medina]{Carolina Medina}
\address{Unidad Acad\'emica de Matem\'aticas. Universidad Aut\'onoma de Zacatecas, Mexico.}
\email{carolitomedina@gmail.com}

\author[Ram\'\i rez]{Santino Ram\'\i rez}
\address{Instituto de F\'\i sica, Universidad Aut\'onoma de San Luis Potos\'{\i}, Mexico.} 
\email{santinormz@if.uaslp.mx}

\author[Ram\'\i rez-Alfons\'\i n]{Jorge L.~Ram\'\i rez-Alfons\'\i n}
\address{IMAG, Univ. Montpellier, CNSR, Montpellier, France.}
\email{\tt jorge.ramirez-alfonsin@umontpellier.fr}

\author[Salazar]{Gelasio Salazar}
\address{Instituto de F\'\i sica, Universidad Aut\'onoma de San Luis Potos\'{\i}, Mexico.} 
\email{\tt gsalazar@ifisica.uaslp.mx}


\date{\today}

\begin{abstract}
An arrangement of pseudocircles $\mathcal A$ is a collection of Jordan curves in the plane that pairwise intersect (transversally) at exactly two points. How many non-equivalent links have $\mathcal A$ as their shadow? Motivated by this question,  we study the number of non-equivalent positive oriented links that have an arrangement of pseudocircles as their shadow. We give sharp estimates on this number when $\mathcal A$ is one of the three unavoidable arrangements of pseudocircles.
\end{abstract}

\maketitle

\section{Introduction}\label{sec:intro}

\DeclareRobustCommand*{\ora}{\overrightarrow}

{Let $L$ be a link. A {\em link diagram} of $L$ is a regular projection of $L$ into $\mathbb{R}^2$ such that the projection of each component is smooth and at most two curves intersect at a point. At each crossing point of the link diagram the curve which goes over the other is specified. A {\em shadow} of a link diagram is {the plane graph obtained by ignoring the over/under passes, turning them into degree $4$ vertices.} We refer the reader to \cite{Colin04} for standard background on knot theory.}

Suppose that $S$ is a link shadow. In general, it is virtually {impossible} to tell exactly which link is being projected to $S$. {What can we say about the links that could be projected to $S$?}

Several variants of this general question have been investigated in the literature. For instance,~\cite{ptaniyama,hanaki2015,taniyamaknots,taniyamalinks,santino,takimura2018} revolve around the {following} question: {\em given a link $L$, which shadows are projections of $L$?}

The latter is closely related to the notion of {\em fertility}: which/how many knots can be obtained from a minimal shadow of a knot~\cite{fertilitycantarella,fertilityito,fertilityhanaki,hanaki2023}. 
Another related problem asks for the construction of ``small'' shadows that are the projection of all knots of a given crossing number~\cite{evenzohar}. {A connection between shadows and the well-known {\em unknotting number} of a link diagram has been treated in \cite{MedinaRamirezSalazar19}. We refer the reader to~\cite{huhtaniyama,medina1,hanaki2010,itotakimura} for further related problems.}

The concept of fertility mentioned above is closely related to what we call the {\em prolificity} of a link shadow $S$, which is simply the number of non-equivalent links that project to $S$. For instance, a shadow of a torus knot $T_{2,n}$ with exactly $n$ crossing points is rather poor in this regard, as its prolificity is only $\lfloor{n/2}\rfloor$~\cite{fertilitycantarella}.

Needless to say, estimating the prolificity of an arbitrary shadow $S$ with $n$ crossings is a daunting task. {A naive} general algorithm to calculate this number consists of generating all $2^n$ diagrams that project to $S$, and then to try to find out by some means how many non-equivalent links {(this is the hard part)} arise from these $2^n$ diagrams. 

{In this direction, we shall}  shed light into the problem by investigating particular families of shadows. In this paper we {focus our attention} on shadows that are arrangements of pseudocircles. We recall that an {\em arrangement of pseudocircles of size $n$} is a collection of $n$ Jordan curves in the plane that pairwise intersect at exactly two points, at which they cross. In Figure~\ref{fig:f30}(a) we illustrate an arrangement of pseudocircles of size $3$. If we assign an orientation to each pseudocircle, we obtain an {\em oriented arrangement of pseudocircles}. For instance, in Figure~\ref{fig:f30}(b) and (d) we illustrate oriented arrangements obtained from the arrangement in (a). For brevity, throughout this work we refer to an (oriented or not) arrangement of pseudocircles simply as an {\em arrangement}.


\def\inca{{\Scale[1.5]{\text{\rm (a)}}}}
\def\incb{{\Scale[1.5]{\text{\rm (b)}}}}
\def\incc{{\Scale[1.5]{\text{\rm (c)}}}}
\def\incd{{\Scale[1.5]{\text{\rm (d)}}}}
\def\ince{{\Scale[1.5]{\text{\rm (e)}}}}

\begin{figure}[ht!]
\centering
\scalebox{0.5}{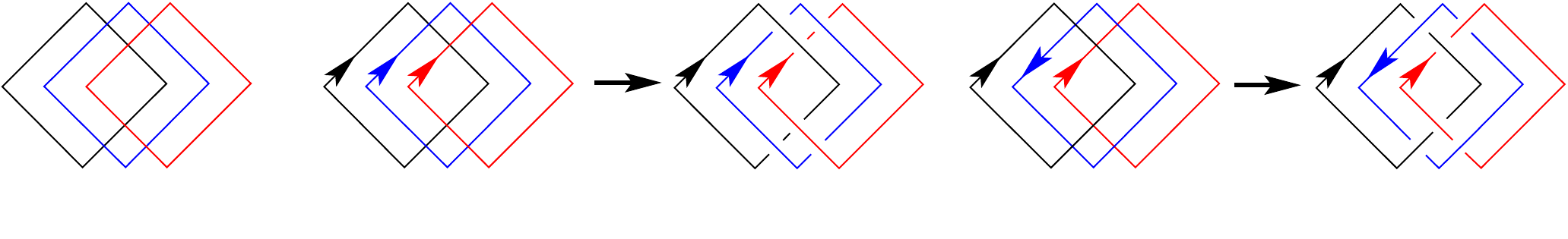}
\caption{In (a) we have an arrangement of three pseudocircles, and in (b) and (d) we have oriented arrangements obtained from (a). In (c) (respectively, (e)) we illustrate the positive link induced by the oriented arrangement in (b) (respectively, (d)).}
\label{fig:f30}
\end{figure}

\subsection{The main question}
Estimating the prolificity of an arbitrary arrangement $\aa$ seems to be far from reach. {We shall thus concentrate} our efforts on estimating the number of non-equivalent {\em positive} links that project to $\aa$. We investigate the following.

\begin{question}\label{que:main}
Let $\aa$ be an unoriented arrangement of pseudocircles. How many non-equivalent positive links project to $\aa$?
\end{question}

{Throughout this paper we investigate Question \ref{que:main} for {\em oriented} links {(as we explain in Section~\ref{sec:concludingremarks}, our results for oriented links imply corresponding results for unoriented links)}.  A link is {\em oriented} if an orientation for each of its connected components is fixed.
Two oriented links $L$ and $M$ are {\em equivalent} if there is an ambient isotopy that takes $L$ to $M$, preserving the orientation of each component. We use $L\sim M$ to denote that $L$ and $M$ are equivalent oriented links.}

\vglue 0.3 cm
\noindent{\bf Remark. }{\sl With the exception of Section~\ref{sec:concludingremarks}, throughout this paper all links under consideration are implicitly assumed to be oriented.}
\vglue 0.3 cm

{As illustrated in Figure \ref{fig:f30}}, each oriented arrangement naturally induces a positive link. We recall that each crossing in a link diagram is either {\em positive} or {\em negative}, according to the convention illustrated in Figure~\ref{fig:fig2}, and a link is {\em positive} if all its crossings are positive. Thus in order to obtain a positive link from an oriented arrangement, one gives to each crossing in the arrangement the over/under assignment that yields a positive crossing. In Figure~\ref{fig:f30}(c) (respectively, (e)) we illustrate the positive link induced by the oriented arrangement (b) (respectively, (d)).

\begin{figure}[ht!]
\begin{center}
\scalebox{0.5}{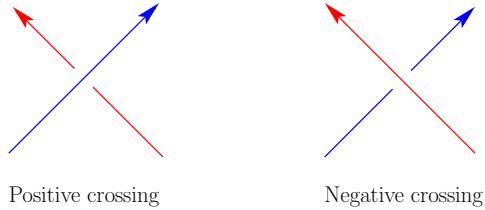}
\end{center}
\caption{Each crossing in a link diagram is either positive or negative, according to this convention.}
\label{fig:fig2}
\end{figure}

We {denote by} $\linksor{\aa}$ the collection of all positive oriented links that project to $\aa$. Thus if $\aa$ is an unoriented arrangement of $n$ pseudocircles, each of the $2^n$ ways to orient the $n$ pseudocircles in $\aa$ induces an oriented arrangement, which in turn naturally induces a link in $\linksor{\aa}$. Conversely, it is easy to see that each  link $L$ in $\linksor{\aa}$ naturally induces an oriented arrangement: the orientations of the components of $L$ naturally yield orientations of the pseudocircles in $\aa$. 

There is thus a one-to-one correspondence between the  links in $\linksor{\aa}$ and the $2^n$ distinct ways to orient the $n$ pseudocircles in $\aa$, that is, the $2^n$ distinct oriented arrangements that have $\aa$ as its underlying unoriented arrangement. Therefore $|\linksor{\aa}|=2^n$. We let $\numbor{\aa}$ denote the number of non-equivalent links in $\linksor{\aa}$ (that is, the number of equivalence classes in $\linksor{\aa}$).


For instance, in Figure~\ref{fig:f90} we illustrate the $2^3=8$ positive links induced by the arrangement $\aa$ in Figure~\ref{fig:f30}. As we also illustrate in that figure, it is easy to verify that the six links on the left-hand side of that figure are equivalent to each other, and the two links on the right-hand side are equivalent to each other (but not to the other six). Hence the collection of $8$ links in $\linksor{\aa}$ is partitioned into two equivalence classes. Thus in this particular case $|\linksor{\aa}|=8$ and $\numb{\aa}=2$. 

\begin{figure}[ht!]
\centering
\scalebox{0.47}{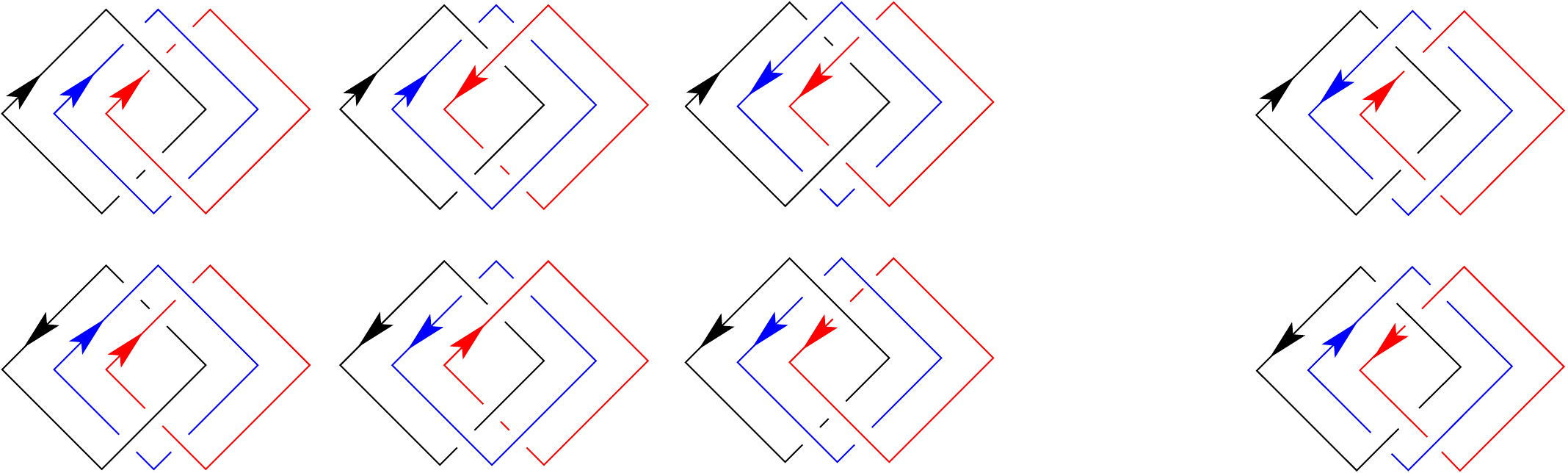}
\caption{The $2^3=8$ positive links induced by the arrangement $\mathscr{A}$ in Figure~\ref{fig:f30} are partitioned into two equivalence (ambient isotopic) classes. Indeed, it is easy to check that the six links on the left-hand side are equivalent to each other, that the two links on the right-hand side are equivalent, and that no link in the first collection is equivalent to a link in the second collection. Therefore $\llbracket{\protect\overrightarrow{\mathcal{L}}^+(\mathscr{A})}\rrbracket=2$.}
\label{fig:f90}
\end{figure}

{We focus our attention on the oriented version of}  Question~\ref{que:main}: 

\begin{question}[Oriented version of Question~\ref{que:main}]\label{que:main2}
Let $\aa$ be an unoriented arrangement of pseudocircles. How many non-equivalent positive oriented links project to $\aa$? Using our notation: how large is $\numbor{\aa}$?
\end{question}


\subsection{Our main results}

Needless to say, the answer to Question~\ref{que:main2} depends on the arrangement $\aa$ under consideration. We investigate this question for three important families of arrangements, namely the {\em unavoidable} arrangements: the ring arrangement, the boot arrangement and the flower arrangement. Let us quickly recall these families.

In Figure~\ref{fig:fig270} we illustrate three arrangements of size $6$: the {\em ring} arrangement $\rr_6$, the {\em boot} arrangement $\bb_6$, and the {\em flower} arrangement $\ff_6$. It is straightforward to generalize these arrangements to obtain arrangements $\rr_n,\bb_n$, and $\ff_n$, for any positive integer $n$. Note that $\rr_1{=}\bb_1{=}\ff_1$ and $\rr_2{=}\bb_2{=}\ff_2$, as up to isomorphism there is evidently only one arrangement of size $1$ and only one arrangement of size $2$.

\begin{figure}[ht!]
\centering
\scalebox{0.7}{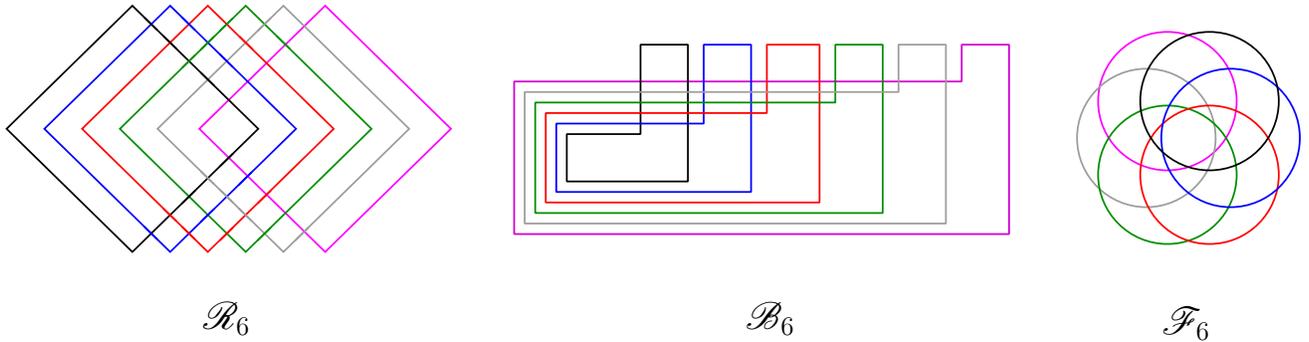}
\caption{From left to right, the {\em ring} arrangement $\rr_6$, the {\em boot} arrangement $\bb_6$, and the {\em flower} arrangement $\ff_6$. These are the unavoidable arrangements of size $6$.}
\label{fig:fig270}
\end{figure}

It was proved in~\cite{pams} that for each integer $n\ge 1$, $\rr_n,\bb_n$, and $\ff_n$ are the three {\em unavoidable} arrangements of pseudocircles, in the following sense:

\begin{theorem}[{\cite[Theorem 2]{pams}}]
For each fixed $n \ge 1$, every sufficiently large arrangement of pseudocircles has a subarrangement isomorphic to $\rr_n,\bb_n$, or $\ff_n$.
\end{theorem}

These three families of arrangements $\rr_n,\bb_n$ and $\ff_n$ are then the {\em unavoidable} families, in the sense that these are the (only) arrangements that are guaranteed to exist as ``large'' subarrangements of an arbitrary arrangement. {The latter is the analogue to the seminal Erd{\H{o}}s-Szekeres  Theorem \cite{ErdosSzkeres35} asserting that  for each fixed $m\ge 1$, every sufficiently large simple arrangement of pseudolines in $\mathbb{R}\mathbb{P}^2$ has a {\em cyclic} subarrangement of size $m$.}

We investigate Question~\ref{que:main2}  for the three unavoidable families of arrangements. Not surprisingly, determining the exact values of $\numbor{\rr_n},\numbor{\bb_n}$, and $\numbor{\ff_n}$ for small values of $n$ is a cumbersome task that relies on a case analysis, and so it makes more sense to investigate these numbers for large values of $n$, that is, the asymptotic behaviour of these numbers. More specifically, we are interested in the growth of $\numbor{\rr_n},\numbor{\bb_n}$, and $\numbor{\ff_n}$ relative to $2^n$, which is the total number of links in $\linksor{\rr_n},\linksor{\bb_n}$, and $\linksor{\ff_n}$.

Our main {contributions} are exact asymptotic estimates for $\numbor{\rr_n}, \numbor{\bb_n}$, and $\numbor{\ff_n}$. For the reader not familiar with this notation, we recall that $o(n)$ stands for a function that goes to $0$ as $n$ goes to infinity.

\begin{theorem}[The number of positive oriented links that project to $\rr_n$]\label{thm:arrr}
\[
\numbor{\rr_n} = \biggl(\frac{1}{4} + o(n)\biggr) \cdot 2^n.
\]
\end{theorem}

\begin{theorem}[The number of positive oriented links that project to $\bb_n$]\label{thm:arrs}
\[
\numbor{\bb_n} = \biggl({1} + o(n)\biggr) \cdot 2^n.
\]
\end{theorem}

\begin{theorem}[The number of positive oriented links that project to $\ff_n$]\label{thm:arrf}
\[
\numbor{\ff_n} = \biggl(\frac{1}{{2n}} + o(n)\biggr) \cdot 2^n.
\]
\end{theorem}


At a high level, to prove Theorem~\ref{thm:arrr} we show that if we take a random {oriented} link in $\linksor{\rr_n}$ then with high probability its equivalence class in $\linksor{\rr_n}$ has size exactly four. Similarly, to prove Theorem~\ref{thm:arrs} (respectively, Theorem~\ref{thm:arrf}) we show that if we take a random {oriented} link in $\linksor{\bb_n}$ (respectively, in $\linksor{\ff_n}$) then with high probability its equivalence class in $\linksor{\bb_n}$ (respectively, $\linksor{\ff_n}$) has size exactly one (respectively, ${2n}$).




The paper is organized as follows. In next section we briefly discuss some needed background.

In Section~\ref{sec:proofthmarrr}, we set the strategy to prove Theorem~\ref{thm:arrr} concerning ring links. We reduce Theorem~\ref{thm:arrr} to Proposition~\ref{pro:arrr}. In Section~\ref{sec:proofproparrr} Proposition~\ref{pro:arrr} is in turn reduced to Lemmas~\ref{lem:workhorsering} and \ref{lem:oscrings2}. Lemma~\ref{lem:workhorsering} is proved in Section~\ref{sec:proofworkhorsering} while Lemma~\ref{lem:oscrings2} is treated in Section~\ref{sec:smallringlinks} (for small values by using intrinsic symmetry groups) and in Section~\ref{sec:proofoscrings2} for arbitrary values.

In Section~\ref{sec:proofarrs}, we set the strategy to prove Theorem~\ref{thm:arrs} concerning boot links. We reduce Theorem~\ref{thm:arrs} to Proposition~\ref{pro:arrs}. Similarly as for ring links, this proposition is in turn reduced in Section~\ref{sub:proofboots} to Lemmas~\ref{lem:workhorseboot} and \ref{lem:oscboots2}. Lemma~\ref{lem:workhorseboot} is also proved in Section ~\ref{sub:proofboots}, while Lemma~\ref{lem:oscboots2} is proved in Section~\ref{sec:lemboots}.

In Section~\ref{sec:proofarrff} we set the strategy to prove Theorem~\ref{thm:arrf} concerning flower links. This family is trickier, and it needs a few further notions and arguments. The overall strategy is analogous to the one used for ring links and boot links: the theorem is reduced to Proposition~\ref{pro:arrf}, which is in turn reduced to Lemmas~\ref{lem:workhorseflower} and \ref{lem:oscflowers2} in Section~\ref{sec:reduceflowerproposition}. Lemma~\ref{lem:workhorseflower} is also proved in Section~\ref{sec:reduceflowerproposition} while Lemma~\ref{lem:oscflowers2} is treated in Section~\ref{sec:oscillatingflower} (for small values by using intrinsic symmetry groups) and in Section~\ref{sec:proofoscflowers2} for general values.

Finally, in Section~\ref{sec:concludingremarks} we discuss the corresponding versions of Theorems~\ref{thm:arrr},~\ref{thm:arrs}, and~\ref{thm:arrf} for unoriented links.



\section{Basic notation and terminology on isotopies} 
We close this introductory section with some remarks on the notation and terminology that we use throughout this paper. 

For brevity, we will refer to an ambient isotopy simply as an {\em isotopy}. We use $\Isot{\ii}{L}{M}$ to denote that $\ii$ is an isotopy that takes a link $L$ to a link $M$, and say that $\ii$ is an {\em $\isot{L}{M}$ isotopy}.

Many of our arguments will involve permutations induced by isotopies. Before we explain this in more detail, let us lay out the terminology we shall use for permutations. {As usual, the set $\{1,\dots , n\}$ will be denoted by $[n]$}. We use $\identity$ to denote the identity permutation on a set $[n]$, for some positive integer $n$. We use $\reverse$ to denote the {\em reverse} permutation on a set $[n]$, that is, $\reverse(i)=n-i+1$ for $i=1,\ldots,n$. We shall use the one-line notation for permutations. That is, we use $(\pi(1)\, \pi(2)\, \cdots\, \pi(n) )$ to denote the permutation 
$\left(\begin{smallmatrix} 
1 & 2 & \cdots & n\\
\pi(1) & \pi(2) & \cdots & \pi(n) 
\end{smallmatrix}\right)$.
Thus, for instance, the identity permutation $\identity$ on $[n]$ is $(1\,2\,\cdots n)$, and the reverse permutation $\reverse$ on $[n]$ is $(n\, \cdots 2\, 1)$.

As we shall see, the components of every link under consideration will be naturally ordered. Suppose that the components of a link $L$ (respectively, $M$) have a natural order $L_1,\ldots,L_n$ (respectively, $M_1,\ldots,M_n$). Then each $\isot{L}{M}$ isotopy $\ii$ naturally induces a permutation $\pi$ of $[n]=\{1,2,\ldots,n\}$, where $\pi(i)$ is the integer such that $\Isot{\ii}{L_i}{M_{\pi(i)}}$. We say that $\pi$ {\em is the $({L},{M})$-permutation under} $\ii$. We also say that $\ii$ is an $\isop{L}{M}{\pi}$ {\em isotopy}, and we write $\Isop{\ii}{L}{M}{\pi}$. 

In the particularly important case in which $\pi$ is the identity permutation $\identity$ (that is, $\ii$ takes the $i$-th component $L_i$ of $L$ to the $i$-th component $M_i$ of $M$, for $i=1,\ldots,n$), we say that $\ii$ is a {\em strong} $\isot{L}{M}$ {\em isotopy}. 

We finish the section with an elementary fact that will be used frequently in our discussions. 

\begin{observation}\label{obs:comprings} {Let $L,M$ and $N$ be three links and let $\Isot{\ii}{L}{M}$ and $\Isot{\jj}{M}{N}$ be two isotopies. Let} $\pi,\tau$ be permutations such that $\Isop{\ii}{L}{M}{\pi}$ and $\Isop{\jj}{M}{N}{\tau}$. Then $\jj\circ \ii$ is an isotopy such that $\Isop{\jj\circ\ii}{L}{N}{\tau\,\circ\,\pi}$. 
\end{observation}

\section{Ring links: proof of Theorem~\ref{thm:arrr}}\label{sec:proofthmarrr}

For simplicity, throughout this paper we refer to a link in $\linksor{\rr_n}$ as a {\em positive ring link of size $n$} or simply as a {\em ring link of size $n$}, since all links under consideration are positive.

\subsection{Correspondence between ring links and binary words}

A ring link of size $n$ is naturally associated to a binary word of size $n$. Indeed, each link in $\linksor{\rr_n}$ is obtained by assigning an orientation to each of the $n$ pseudocircles in $\rr_n$. Such an orientation assignment can be naturally encoded as follows. See Figure~\ref{fig:f180} for an illustration.

\def\so#1{{\Scale[1.8]{#1}}}
\def\sp#1{{\Scale[1.1]{#1}}}
\def\sq#1{{\Scale[1.3]{#1}}}
\def\sr#1{{\Scale[1.3]{#1}}}

\begin{figure}[ht!]
\centering
\scalebox{0.6}{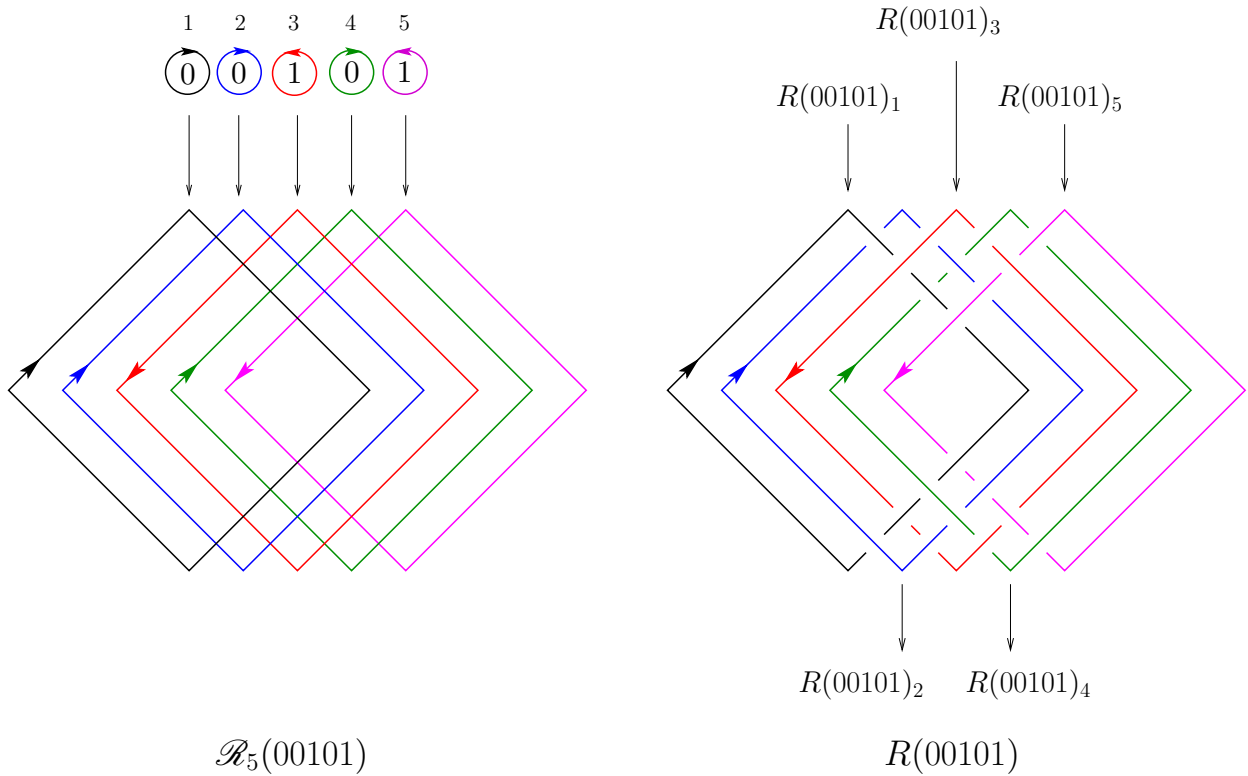}
\caption{The oriented arrangement $\rr_5(00101)$, and its induced positive link, the ring link $R(00101)$. We indicate the five components $R(00101)_1$, $R(00101)_2$, $R(00101)_3$, $R(00101)_4$ and $R(00101)_5$ of $R(00101)$.}
\label{fig:f180}
\end{figure}

Let $i=1,2,\ldots,n$ be the pseudocircles in $\rr_n$, ordered from left to right. We encode an orientation using a word $a$ of length $n$, where the $i$-th entry of the word $a$ is $0$ (respectively, $1$) if pseudocircle $i$ is oriented clockwise (respectively, counterclockwise). Given such a word $a$, we use $\rr_n(a)$ to denote the corresponding oriented arrangement. Finally, we use $R(a)$ to denote the positive link {induced} by $\rr_n(a)$. It seems worth emphasizing that $R$ stands for ``ring''.

For instance, if we take the arrangement $\rr_5$ and orient clockwise pseudocircles $1,2$, and $4$, and orient counterclockwise pseudocircles $3$ and $5$, then we obtain the oriented arrangement $\rr_5(00101)$, illustrated on the left hand side of Figure~\ref{fig:f180}. This oriented arrangement induces the ring link $R(00101)$, illustrated on the right hand side of that figure. As we also illustrate in that figure, for each $i=1,\ldots,n$ we use $R(a)_i$ to denote the $i$-th component of the ring link $R(a)$.

In short, {\em there is a one-to-one correspondence between the set of all binary words of length $n$ and the collection $\linksor{\rr_n}$ of all ring links of size $n$}. As we shall see, a totally analogous observation holds for boot links of size $n$ and also for flower links of size $n$. For the rest of the paper, we shall refer to a binary word simply as a {\em word}.

We finally introduce a few elementary notions that will be used throughout this paper. If $a=a_1\cdots a_n$ is a word, the {\em reverse} $\In{a}$ of $a$ is $a_n\cdots a_1$, and its {\em negation} $\Ov{a}=\Ov{a_1 \cdots a_n}$ is $\Ov{a_1} \cdots \Ov{a_n}$ (as usual, $\Ov{0}=1$ and $\Ov{1}=0$). Thus if $a=00101$, then $\In{a}=10100$ and $\Ov{a}=11010$. Note that every word $a$ satisfies that $\In{(\Ov{a})}=\Ov{\In{a}}$.

A {\em subword} of a word $a_1 a_2 \cdots a_n$ is a word of the form $a_{i_1} a_{i_2} \cdots a_{i_m}$, where $i_1 < i_2 < \cdots < i_m$. A word is {\em oscillating} if it contains neither two consecutive $0$s nor two consecutive $1$s. We note that in the literature a word that we call oscillating would be called ``alternating'', but in the context of knot theory it seems best to avoid this terminology for anything other than its usual meaning.

The opposite notion of an oscillating word is a {\em monotone} word, which consists only of $0$s or only of $1$s. The {\em rank} of a word $a$, denoted by $\rank{a}$, is the length of a longest oscillating subword contained in $a$. Thus, for instance, $\rank{000100}=3$ and $\rank{00100010}=5$. 

If $m$ is a positive integer, we use $0^m$ to denote the monotone word $0\cdots 00$ of length $m$ with only $0$s. Similarly, $1^m$ is the monotone word $1\cdots 11$ of length $m$ with only $1$s.

Note that if $a_{i_1}\cdots a_{i_r}$ is a longest oscillating subword of $a$, then there exist positive integers $\alpha(1),\ldots,\alpha(r)$, such that $a$ is the concatenation $a_{i_1}^{\alpha(1)} \cdots a_{i_r}^{\alpha(r)}$ of $r$ monotone words. We say that $A_1=a_{i_1}^{\alpha(1)},\ldots, A_r=a_{i_r}^{\alpha(r)}$ are the {\em canonical subwords} of $a$, and that $a=a_1^{\alpha(1)} \cdots a_r^{\alpha(r)}=A_1\cdots A_r$ is the {\em canonical decomposition} of $a$. For instance, the word $a=001000110$ has rank $r=5$, and its canonical subwords are $A_1=0^2, A_2=1^1$, and $A_3=0^3$, $A_4=1^2$, and $A_5=0^1$, as schematized in Figure~\ref{fig:fcandec}.

\begin{figure}[ht!]
\centering
\[
    \overbrace{
    \underbrace{00}_{\text{$A_1$}} \text{\hglue 0.001 cm}
    \underbrace{1}_{A_2} \text{\hglue 0.001 cm}
    \underbrace{000}_{A_3} \text{\hglue 0.001 cm}
    \underbrace{11}_{A_4} \text{\hglue 0.001 cm}
    \underbrace{0}_{A_5} \text{\hglue 0.001 cm}
   }^\text{\Scale[1.6]{a=001000110}}
 \]
\caption{The canonical decomposition of $a=0010001100$ is $a=A_1A_2A_3A_4A_5$, where $A_1=00=0^2, A_2=1=1^1, A_3=000=0^3, A_4=11=1^2$, and $A_5=0=0^1$.}
\label{fig:fcandec}
\end{figure}

\subsection{Isotopies that act naturally on ring links}\label{sub:isotopiesring}

The heart of the proof of Theorem~\ref{thm:arrr} is an identification of which ring links are equivalent to a given ring link $R(a)$.  As we shall see shortly (see Observation~\ref{obs:ifring}), $R(a)$ is always equivalent to $R(\Ov{a})$, to $R(\In{a})$, and to $R(\IO{a})$ (and of course, to $R(a)$ itself). This will be the easy part of the identification of the links that are equivalent to $R(a)$. The considerably more difficult part (see Proposition~\ref{pro:arrr}) will be to show that {\em only} these links are equivalent to $R(a)$. 

Let us start by showing that $R(a)$ is always equivalent to $R(\Ov{a})$. We refer the reader to Figure~\ref{fig:f560}, where we illustrate an isotopy $\Ho$ (which stands for ``horizontal'', as $\Ho$ is a rotation of 180 degrees around a horizontal axis lying on the plane of the diagram). As we illustrate in that figure for the particular case in which $a=000101$, if we apply $\Ho$ to a ring link $R(a)=R(a_1\cdots a_n)$, as a result we obtain the ring link $R(\Ov{a_1\cdots a_n})=R(\Ov{a})$. That is, $\Isot{\Ho}{R(a)}{R(\Ov{a})}$. We conclude that (R1) {\em if $a$ is any word, then $R(a)\sim R(\Ov{a})$}.

\def\tf#1{{\Scale[1.9]{#1}}}
\def\tg#1{{\Scale[5]{#1}}}
\def\th#1{{\Scale[3]{#1}}}
\begin{figure}[ht!]
\centering
\scalebox{0.5}{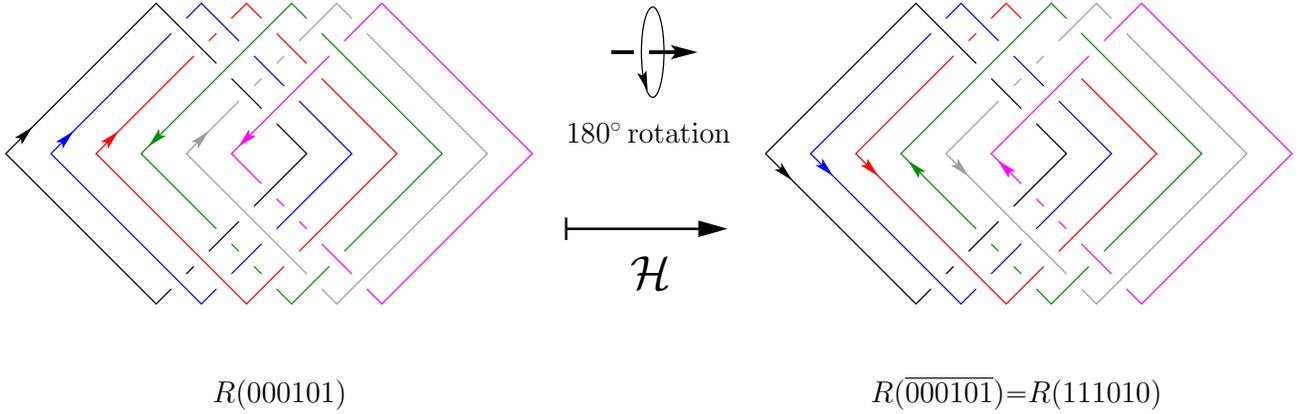}
\caption{If we apply the isotopy $\Ho$ to the ring link $R(000101)$ we obtain the ring link $R(\Ov{000101})=R(111010)$. The permutation of this $R(000101) \mapsto R(111010)$ isotopy is the identity permutation $\identity$ on $[6]$. In general, if we apply $\Ho$ to a ring link $R(a)$ we obtain the ring link $R(\Ov{a})$, that is, $\Isot{\Ho}{R(a)}{R(\Ov{a})}$. If $R(a)$ has $n$ components then $\Ho$ maps the $i$-th component of $R(a)$ to the $i$-th component of $R(\Ov{a})$, for $i=1,\ldots,n$. That is, $\Isop{\Ho}{R(a)}{R(\Ov{a})}{\identity}$.}
\label{fig:f560}
\end{figure}

Consider now the isotopy  $\Ve$ illustrated in Figure~\ref{fig:f550}. The letter $\Ve$ stands for ``vertical'', as $\Ve$ is a rotation of 180 degrees around a vertical axis on the plane of the diagram. As we illustrate in that figure, if we apply $\Ve$ to a ring link $R(a)=R(a_1\cdots a_n)$, as a result we obtain the ring link $R\bigl(\,\Ov{\In{(a_1\cdots a_n)}}\,\bigr)=R\bigl(\In{(\Ov{a_1\cdots a_n})}\bigr)=R\bigl(\In{(\Ov{a})}\bigr)$. That is, $\Isot{\Ve}{R(a)}{R(\IO{a})}$. Therefore (R2) {\em if $a$ is any word, then $R(a)\sim R(\IO{a})$}.

\def\tf#1{{\Scale[1.9]{#1}}}
\def\tg#1{{\Scale[5]{#1}}}
\def\th#1{{\Scale[3]{#1}}}
\begin{figure}[ht!]
\centering
\scalebox{0.5}{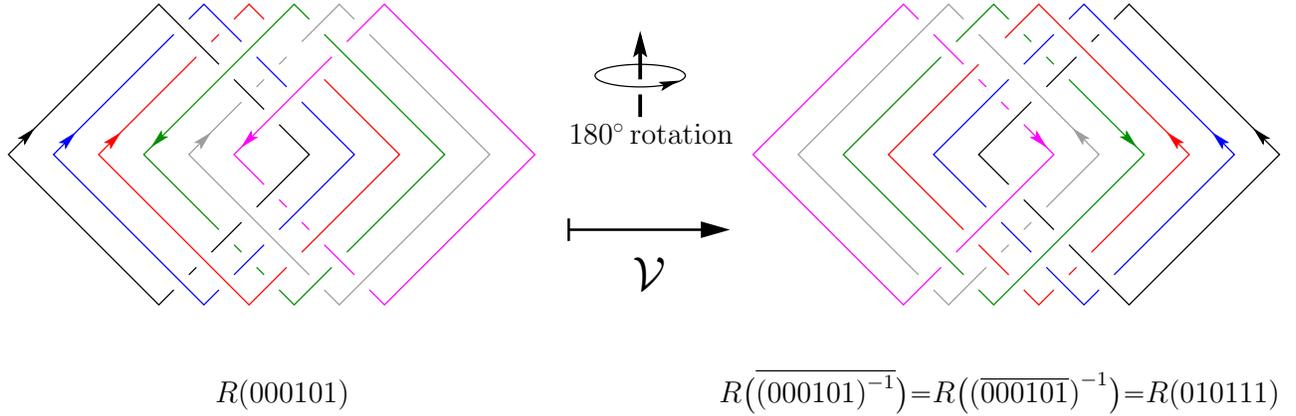}
\caption{If we apply the isotopy $\Ve$ to the ring link $R(000101)$ we obtain the ring link $R(\In{(\Ov{000101})})=R(010111)$. The permutation of this $\isot{R(000101)}{R(010111)}$ isotopy is the reverse permutation $\reverse$ on $[6]$. In general, if we apply $\Ve$ to a ring link $R(a)$ we obtain the ring link $R(\In{(\Ov{a})})$, that is, $\Isot{\Ve}{R(a)}{R(\IO{a})}$. If $R(a)$ has $n$ components then $\Ve$ maps the $i$-th component of $R(a)$ to the $(n-i+1)$-st component of $R(\IO{a})$, for $i=1,\ldots,n$. That is, $\Isop{\Ve}{R(a)}{R(\IO{a})}{\reverse}$.}
\label{fig:f550}
\end{figure}

Since $\Isot{\Ho}{R(a)}{R(\Ov{a})}$ and $\Isot{\Ve}{R(a)}{R(\IO{a})}$, it follows that $R(a)\sim R(\Ov{a})$ and $R(a)\sim R(\IO{a})$. If we apply to $R(a)$ the isotopy $\Ho$ followed by the isotopy $\Ve$ (that is, $\Ve\circ\Ho$), we obtain the ring link $R(\IO{\Ov{a}})=R(\In{a})$. Therefore (R3) {\em if $a$ is any word, then $R(a)\sim R(\In{a})$}. 

We finally note that trivially (R4) {\em if $a$ is any word, then $R(a)\sim R(a)$}. In view of (R1)--(R4) we have the following.

\begin{observation}\label{obs:ifring}
Let $a$ {and} $b$ {be} words. If $b$ is either $a, \Ov{a}, \In{a}$, or $\IO{a}$, then $R(a)\sim R(b)$. 
\end{observation}

\subsection{Reducing Theorem~\ref{thm:arrr} to a proposition}

Given a word $a$, we have thus identified four words $b$ (including $a$ itself) such that $R(a)\sim R(b)$. The main ingredient in the proof of Theorem~\ref{thm:arrr} is that the converse statement also holds, as long as the rank of $a$ is at least four:

\begin{proposition}[Implies Theorem~\ref{thm:arrr}]\label{pro:arrr}
Let $a$ be a word of rank {$r\ge 4$}. If $b$ is a word such that $R(a)\sim R(b)$, then $b$ is either $a, \Ov{a}, \In{a}$, or $\IO{a}$.
\end{proposition}

The proof of this proposition will take considerably more effort than the arguments that led to Observation~\ref{obs:ifring}. We defer the proof of the proposition for the moment, and show that Theorem~\ref{thm:arrr} follows easily by combining these two statements.

\begin{proof}[Proof of Theorem~\ref{thm:arrr} (assuming Proposition~\ref{pro:arrr})]
Let $a$ be a word of length $n$. Using standard calculations one obtains that the probability that $\rank{a}$ is less than $4$ goes to $0$ as $n\to \infty$, and also the probability that there are two identical words in $\{a,\Ov{a},\In{a},\In{(\Ov{a})}\}$ also goes to $0$ as $n\to \infty$. By Observation~\ref{obs:ifring} and Proposition~\ref{pro:arrr} this implies that the probability that there are {\em exactly} four distinct words $b$ of length $n$ (including $a$) such that $R(a) \sim R(b)$ goes to $1$ as $n\to\infty$.

The one-to-one correspondence between binary words of length $n$ and elements of $\linksor{\rr_n}$ then implies that the probability that the equivalence class of a random link in $\linksor{\rr_n}$ has size $4$ goes to $1$ as $n \to \infty$. Since $|\linksor{\rr_n}|=2^n$, Theorem~\ref{thm:arrr} follows. 
\end{proof}

\section{Proof of Proposition~\ref{pro:arrr}}\label{sec:proofproparrr}

Before we move on to the proof of Proposition~\ref{pro:arrr} (or, more accurately, to reducing the proposition to a couple of lemmas, namely Lemmas~\ref{lem:workhorsering} and~\ref{lem:oscrings2}), {we} briefly discuss sublinks of ring links, as they play a central role in this discussion.

\subsection{Sublinks of ring links}

Let $a=a_1\ldots a_n$ be a  word, and let $i_1,\ldots,i_k$ be integers such that $1 \le i_1 < \cdots < i_k \le n$. Then $a_{i_1}\cdots a_{i_k}$ is a subword of $a$, and this subword naturally corresponds to a link $R(a)_{i_1}\cup \cdots \cup R(a)_{i_k}$ (recall that $R(a)_i$ is the $i$-th component of the ring link $R(a)$). We say that $R(a)_{i_1}\cup \cdots \cup R(a)_{i_k}$ is a {\em sublink} of $R(a)$, and for brevity we use $R(a)_{i_1,\ldots,i_k}$ to denote it.

It is worth noting that this notation is consistent with the way we denote a single component of $R(a)$: if $k=1$ then we have a single integer $i_1$, and so the corresponding sublink consists of the component $R(a)_{i_1}$.

We say that the ring link $R(a)$ is {\em oscillating} if the word $a$ is oscillating, and we say that the sublink $R(a)_{i_1,\ldots,i_k}$ of $R(a)$ is {\em oscillating} if $a_{i_1} \cdots a_{i_k}$ is an oscillating subword of $a$. Note that obviously no oscillating sublink of $R(a)$ can have size larger than the rank $r$ of $a$, since this is the length of a longest oscillating subword of $a$. 

\subsection{Reducing Proposition~\ref{pro:arrr} to two lemmas}

Recall that Proposition~\ref{pro:arrr} claims that if $a$ is a word with rank $r\ge 4$, and $R(a)\sim R(b)$, then $b$ is either $a,\Ov{a},\In{a}$, or $\IO{a}$.

The proof of this proposition has two main ingredients. The first one is that if $\ii$ is an $\isot{R(a)}{R(b)}$ isotopy, then we can fully reconstruct $b$ from the action of $\ii$ on an oscillating sublink $R(a)_{i_1,\ldots,i_r}$ of $R(a)$ of size $r$. More precisely, let $R(b)_{j_1,\ldots,j_r}$ be the image of $R(a)_{i_1,\ldots,i_r}$ under $\ii$. Suppose that we know 

(i) the subword $b_{j_1} \cdots b_{j_r}$ of $b$, and 

(ii) the $(R(a)_{i_1,\ldots,i_r},R(b)_{j_1,\ldots,j_r})$-permutation under $\ii$. 

Then from (i) and (ii) we can fully determine the word $b$. This is the content of Lemma~\ref{lem:workhorsering} below.

The second ingredient in the proof of Proposition~\ref{pro:arrr} is that we can actually fully understand all the possibilities for (i) and (ii) in the previous paragraph. Indeed, as we claim in Lemma~\ref{lem:oscrings2}, $b_{j_1} \cdots b_{j_r}$ is necessarily also oscillating (as $a_{i_1} \cdots a_{i_r}$) and the $(R(a)_{i_1,\ldots,i_r},R(b)_{j_1,\ldots,j_r})$-permutation under $\ii$ is either the identity permutation $\identity$ or the reverse permutation $\reverse$. As we shall see, Proposition~\ref{pro:arrr} follows easily by combining these two lemmas.

Lemma~\ref{lem:workhorsering} involves the concept of the $\pi$-image of a word. We refer the reader to Figure~\ref{fig:f610} for an illustration of this crucial notion. Let $a,b$ be words of the same rank $r$, and let $A_1A_2\cdots A_r$ and $B_1 B_2 \cdots B_r$ be the canonical decompositions of $a$ and $b$, respectively. Let $\pi$ be a permutation of $[r]$. We say $b$ is the $\pi$-{\em image} of $a$ if there is a bijection between $A_i$ and $B_{\pi(i)}$, that is, if $|A_i|=|B_{\pi(i)}|$ for $i=1,\ldots,r$. 

\def\ya#1{{\Scale[2]{#1}}}
\def\tf#1{{\Scale[1.9]{#1}}}
\def\tg#1{{\Scale[5]{#1}}}
\def\th#1{{\Scale[3]{#1}}}
\def\tj#1{{\Scale[3]{#1}}}
\begin{figure}[ht!]
\centering
\scalebox{0.5}{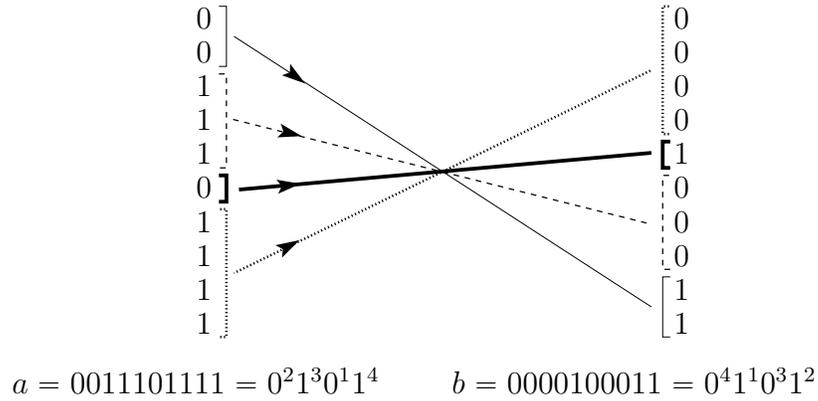}
\caption{On the left-hand side we have the word $a=0011101111$, whose canonical decomposition is $A_1A_2A_3A_4$, where $A_1=00=0^2, A_2=111=1^3, A_3=0=0^1$, and $A_4=1111=1^4$. On the right-hand side we have the word $b=0000100011$, whose canonical decomposition is $B_1B_2B_3B_4$, where $B_1=0000=0^4, B_2=1=1^1, B_3=000=0^3$, and $B_4=11=1^2$. We indicate with arrows a natural bijection between the canonical subwords of $a$ and the canonical subwords of $b$. If we let $\pi$ be the reverse permutation $\reverse=(4\,3\,2\,1)$ on $[4]$, we have that $A_i$ is mapped to (has the same length as) $B_{\pi(i)}$ for $i=1,2,3,4$. Thus $b$ is the $\pi$-image of $a$.}
\label{fig:f610}
\end{figure}

Clearly, if $b$ is the $\pi$-image of $a$ for some permutation $\pi$ of $[r]$, then in particular $b$ has the same length as $a$. Moreover, if we know an oscillating subword $b_{j_1}\cdots b_{j_r}$ of $b$ of length $r$ and we know a permutation $\pi$ such that $b$ is the $\pi$-image of $a$, then we can fully reconstruct $b$. This explains the importance of the next statement, whose proof is deferred to the next section.

\begin{lemma}\label{lem:workhorsering}
Let $a$ be a word with rank $r\ge 4$, let $b$ be a word such that $R(a)\sim R(b)$, and let $\ii$ be an $\isot{R(a)}{R(b)}$ isotopy. Let $R(a)_{i_1,\ldots,i_r}$ be an oscillating sublink of $R(a)$, and let $R(b)_{j_1,\ldots,j_r}$ be its image under $\ii$. Then $b$ is the $\pi$-image of $a$, where $\pi$ is the $(R(a_{i_1,\ldots,i_r}),R(b_{j_1,\ldots,j_r}))$-permutation under $\ii$.
\end{lemma}

In order to make use of Lemma~\ref{lem:workhorsering} we need to know which sublinks $R(b)_{j_1,\ldots,j_r}$ of $R(b)$ can be the image of $R(a)_{i_1,\ldots,i_r}$ under $\ii$, and we need to know which are the possible $(R(a_{i_1,\ldots,i_r}),R(b_{j_1,\ldots,j_r}))$-permutations under $\ii$. This is precisely the information given by our next statement, the second main ingredient in the proof of Proposition~\ref{pro:arrr}.

\begin{lemma}\label{lem:oscrings2}
Let $a=a_1 \cdots a_n$ be a word of rank $r\ge 4$, and let $R(a)_{i_1,\ldots,i_r}$ be an oscillating sublink of $R(a)$ of size $r$. Let $b$ be a word such that $R(a)\sim R(b)$, and let $\ii$ be an $\isot{R(a)}{R(b)}$ isotopy. Let $R(b)_{j_1,\ldots,j_r}$ be the sublink of $R(b)$ that is the image of $R(a)_{i_1,\ldots,i_r}$ under $\ii$. Then,
\begin{enumerate}
\item[(1)] the $(R(a)_{i_1,\ldots,i_r},R(b)_{j_1,\ldots,j_r})$-permutation under $\ii$ is either the identity permutation $\identity$ or the reverse permutation $\reverse$; and 
\item[(2)] $b_{j_1}\cdots b_{j_r}$ is an oscillating subword of $b$. That is, $R(b)_{j_1,\ldots,j_r}$ is an oscillating sublink of $R(b)$.
\end{enumerate}
\end{lemma}

We defer the proof of Lemma~\ref{lem:oscrings2} to Sections~\ref{sec:smallringlinks} and~\ref{sec:proofoscrings2}, and we prove Proposition~\ref{pro:arrr} assuming the lemmas. Before proceeding to the proof we note the following easy consequence of Lemma~\ref{lem:oscrings2}.

\begin{corollary}\label{cor:ringsrank}
Let $a$ be a word of rank $r\ge 4$, and let $b$ be a word such that $R(a)\sim R(b)$. Then $\rank{b}=r$.
\end{corollary}

\begin{proof}
Let $s:=\rank{b}$. Lemma~\ref{lem:oscrings2} implies that there is an oscillating sublink of $R(b)$ of size $r$, and so it follows that $s\ge r$. In particular, $\rank{b}\ge 4$, and so we can apply the lemma also to an $\isot{R(b)}{R(a)}$ isotopy, obtaining that there must exist an oscillating sublink of $R(a)$ of size $s$, and so $r\ge s$. Thus $s=r$. 
\end{proof}

\begin{proof}[Proof of Proposition~\ref{pro:arrr} (assuming Lemmas~\ref{lem:workhorsering} and~\ref{lem:oscrings2})]
Let $a=a_1\cdots a_n$ be a word with rank $r\ge 4$. Let $b=b_1\cdots b_n$ be a word such that $R(a)\sim R(b)$, and let $\ii$ be an $\isot{R(a)}{R(b)}$ isotopy. 

Let $R(a)_{i_1,\ldots,i_r}$ be an oscillating sublink of $R(a)$. Let $R(b)_{j_1,\ldots,j_r}$ be the sublink of $R(b)$ that is the image of $R(a)_{i_1,\ldots,i_r}$ under $\ii$, and let $\pi$ be the $(R(a)_{i_1,\ldots,i_r},R(b)_{j_1,\ldots,j_r})$-permutation under $\ii$.

Since $a_{i_1} \cdots a_{i_r}$ is oscillating, it follows that the only two oscillating words of length $r$ are $a_{i_1} \cdots a_{i_r}$ and $\Ov{a_{i_1} \cdots a_{i_r}}$. Therefore Lemma~\ref{lem:oscrings2}\,(2) implies that either 
\begin{enumerate}
\item[($\dag$)] $b_{j_1}\cdots b_{j_r}=a_{i_1}\cdots a_{i_r}$; or
\item[($\ddag$)] $b_{j_1}\cdots b_{j_r}=\Ov{a_{i_1}\cdots a_{i_r}}$. 
\end{enumerate}

Let $a=A_1 \cdots A_r=a_{i_1}^{|A_1|}\cdots a_{i_r}^{|A_r|}$ be the canonical decomposition of $a$. We note that Corollary~\ref{cor:ringsrank} implies that the rank of $b$ is also $r$, and so we can let  $b=B_1 \cdots B_r=b_{j_1}^{|B_1|}\cdots b_{j_r}^{|B_r|}$ be the canonical decomposition of $b$. 

Lemma~\ref{lem:workhorsering} implies that $b$ is the $\pi$-image of $a$, and Lemma~\ref{lem:oscrings2}\,(1) implies that $\pi$ is either $\identity$ or $\reverse$. Therefore either 
\begin{enumerate}
\item[($*$)] $|B_k|=|A_k|$ for $k=1,\ldots,r$; or
\item[($**$)] $|B_k|=|A_{r-k+1}|$ for $k=1,\ldots,r$.
\end{enumerate}

A glance at the canonical decompositions of $a$ and $b$ shows that if ($*$) and ($\dag$) hold then $b=a$; if ($*$) and ($\ddag$) hold then $b=\Ov{a}$; if ($**$) and ($\dag$) hold then $b=\In{a}$; and if ($**$) and ($\ddag$) hold then $b=\IO{a}$. Thus $b$ is either $a,\Ov{a},\In{a}$, or $\IO{a}$.
\end{proof}

\section{Proof of Lemma~\ref{lem:workhorsering}}\label{sec:proofworkhorsering}

The proof of Lemma~\ref{lem:workhorsering} is particularly important, as it will be ``recycled'' almost in its entirety when we deal with boot links and with flower links.

\begin{proof}[Proof of Lemma~\ref{lem:workhorsering} (assuming Lemma~\ref{lem:oscrings2})]
The proof relies on a detailed description of the canonical decompositions of $a$ and $b$. Let $a=A_1\cdots A_r$ be the canonical decomposition of $a$. For each $i=1,\ldots,r$ we let $p_i,q_i$ be the integers such that $A_i=a_{p_i}a_{p_{i+1}} \cdots a_{q_i}$, and let $I_i=\{p_i,p_{i+1},\ldots,q_i\}$. Note that actually $p_i$ may be the same as $q_i$, which is the case when the canonical subword $A_i$ has length $1$. Also note that $p_1=1$ and $q_r=n$. Thus the canonical decomposition of $a$ is as follows: 
\[
\sp{a}\,\,\,\sp{=}\,\,\sp{\underbrace{a_{p_1}\,\, \cdots\,\, a_{q_1}}_{\text{$\sp{A_1}$}}} \,\,\,
\sp{\underbrace{a_{p_2}\,\, \cdots \,\,a_{q_2}}_{\text{$\sp{A_2}$}}} \, \cdots\cdots \,
\sp{\underbrace{a_{p_r}\,\, \cdots\,\, a_{q_r}}_{\text{$\sp{A_r}$}}}\,\sp{.}
\]

We say that $R(a)_{p_k,\ldots,q_k}$ is the {\em $k$-th canonical sublink $R_k$ of $R(a)$}, for $k=1,\ldots,r$. Evidently, $R(a)$ is the disjoint union of its canonical sublinks. A crucial observation is that a sublink of $R(a)$ of size $r$ is oscillating if and only if it contains exactly one component of each canonical sublink.

Now, by Corollary~\ref{cor:ringsrank}, the rank of $b$ is also $r$. Thus we let $b=B_1\cdots B_r$ be the canonical decomposition of $b$. For each $i=1,\ldots,r$ we let $s_i,t_i$ be the integers such that $B_i=b_{s_i} \cdots b_{t_i}$, and let $J_i=\{s_i,s_{i+1},\ldots,t_i\}$. Thus 
\[
\sp{b}\,\,\,\sp{=}\,\,\sp{\underbrace{b_{s_1}\,\, \cdots\,\, b_{t_1}}_{\text{$\sp{B_1}$}}} \,\,\,
\sp{\underbrace{b_{s_2}\,\, \cdots \,\,b_{t_2}}_{\text{$\sp{B_2}$}}} \, \cdots\cdots \,
\sp{\underbrace{b_{s_{i}} \,\,\cdots\,\, b_{t_i}}_{\text{$\sp{B_i}$}}} \cdots\cdots \,
\sp{\underbrace{b_{s_{r}}\,\, \cdots\,\, b_{t_r}}_{\text{$\sp{B_r}$}}}\,\sp{.}
\]

Similarly as with $R(a)$, we say that $R(b)_{s_k,\ldots,t_k}$ is the {\em $k$-th canonical sublink $R_k(b)$ of $R(b)$}, for $k=1,\ldots,r$. Clearly $R(b)$ is the disjoint union of its canonical sublinks, and a sublink of $R(b)$ of size $r$ is oscillating if and only if it contains exactly one component of each canonical sublink.

Note that since $a_{i_1}\cdots a_{i_r}$ is an oscillating subword of $a$, it follows that necessarily $a_{i_k}\in A_k$ for $k=1,\ldots,r$. Thus $A_k=a_{i_k}^{|A_k|}$ for $k=1,\ldots,r$, and so ($*$) {$a=a_{i_1}^{|A_1|}\cdots a_{i_r}^{|A_r|}$. A totally analogous argument shows that ($**$) $b=b_{i_1}^{|B_1|}\cdots b_{i_r}^{|B_r|}$}. 

Recall that $\pi$ is the $(R(a)_{i_1,\ldots,i_k},R(b)_{j_1,\ldots,j_k})$-permutation under $\ii$. This means that $\ii$ maps $R(a)_{i_k}$ to $R(b)_{j_{\pi(k)}}$ for $k=1,\ldots,r$. Since $R(a)_{i_k}$ (respectively, $R(b)_{j_{\pi(k)}}$) is in the canonical sublink $R_k(a)$ of $R(a)$ (respectively, in the canonical sublink $R_{\pi(k)}$ of $R(b)$), this implies that $\ii$ takes one particular component of $R_k(a)$ to one particular component of $R_{\pi(k)}(b)$, for $k=1,\ldots,r$. The key argument is that a much stronger statement holds: ($\dag$) {\em for $k=1,\ldots,r$, $\ii$ takes each component of $R_k(a)$ to a component in $R_{\pi(k)}(b)$.}

To prove ($\dag$), by way of contradiction suppose that there is a $k\in [r]$ such that $\ii$ takes some component $R(a)_{i'_k}$ of $R_k(a)$ to a component in $R(b)$ that is not in $R_{\pi(k)}(b)$. Then the image under $\ii$ of the oscillating sublink $R(a)_{i_1,\ldots,i_{k-1},i'_k,i_{k+1},\ldots,i_r}$ of $R(a)$ does not contain any component of the canonical sublink $R_{\pi(k)}(b)$ of $R(b)$. Therefore $\ii$ takes an oscillating sublink of $R(a)$ of size $r$ to a sublink of $R(b)$ that is not oscillating. But this contradicts Lemma~\ref{lem:oscrings2}(2). Thus ($\dag$) holds.

Since $|R_k(a)|=|A_k|$ and $|R_{\pi(k)}(b)|=|B_{\pi(k)}|$ for $k=1,\ldots,r$, ($\dag$) implies that $|A_k| \le |B_{\pi(k)}|$ for $k=1,\ldots,r$. Since $\sum_{k=1}^r |A_k|=\sum_{k=1}^r |B_{\pi(k)}| = n$, it follows that ($\ddag$) $|A_k| = |B_{\pi(k)}|$ for $k=1,\ldots,r$. Therefore $b$ is the $\pi$-image of $a$.
\end{proof}

\section{Towards the proof of Lemma~\ref{lem:oscrings2}: small ring links}\label{sec:smallringlinks}

We will prove Lemma~\ref{lem:oscrings2} by induction on the length $n$ of $a$. Our aim in this section is to prove the lemma when $n=4$, which is the base case of the induction. As we shall see, this base case is equivalent to Claim~\ref{cla:oscrings2} at the end of the section.

The proof of this claim makes essential use of the intrinsic symmetry groups of the ring links $R(0101)$ and $R(1010)$. In Section~\ref{sub:sym} we give a brief overview of the notion of an intrinsic symmetry group, and in Section~\ref{sub:symsmallring} we identify the intrinsic symmetry groups of these two links. Finally, in Section~\ref{sub:smallring} we establish Claim~\ref{cla:oscrings2}.

\subsection{Intrinsic symmetry groups}\label{sub:sym}


Throughout this paper $\ZZ_2$ is the multiplicative group $\{-1,1\}$, and (as usual) $S_n$ is the group of all permutations of size $n$. If $K$ is any oriented knot then $(-1)\cdot K$ is $K$ with its orientation reversed, whereas $1\cdot K$ is simply $K$ with its given orientation.

In order to introduce the notion of an intrinsic symmetry of a link we refer the reader to Figure~\ref{fig:f650}. At the top of that figure we illustrate how the isotopy $\Ho$ takes the ring link $R(0101)$ to the ring link $R(1010)$. As we show at the bottom of the figure, if we ignore for a moment the orientations of the components of these links, we may consider that $\Ho$ takes the ring link $R(0101)$ {\em to itself}. 

\def\tf#1{{\Scale[1.9]{#1}}}
\def\tg#1{{\Scale[5]{#1}}}
\def\th#1{{\Scale[3]{#1}}}
\def\tj#1{{\Scale[3]{#1}}}
\begin{figure}[ht!]
\centering
\scalebox{0.5}{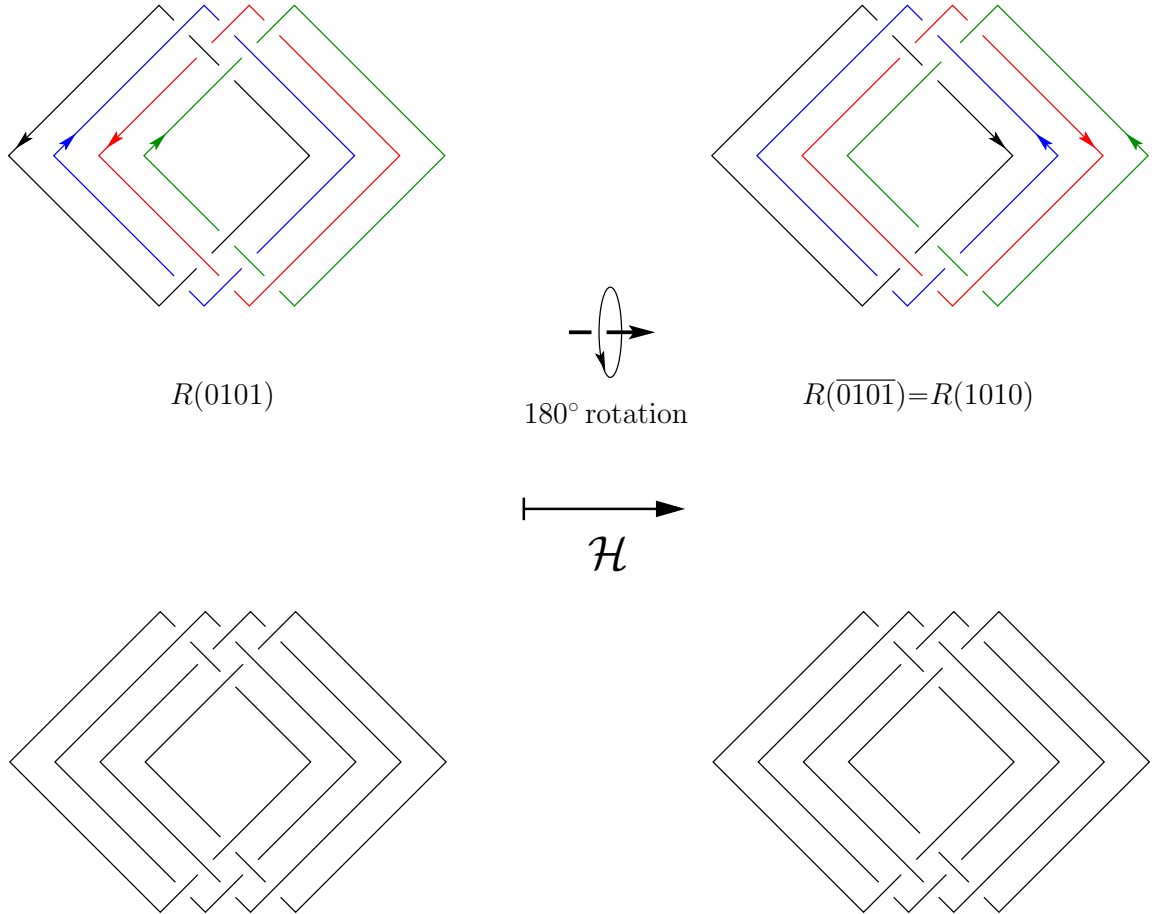}
\caption{An application of the ambient isotopy $\Ho$ takes $R(0101)$ to $R(1010)$. As we show below, if we disregard for a moment the orientation of the components, it is valid to say that $\Ho$ takes $R(0101)$ to itself. Thus we may regard this action of $\Ho$ on $R(0101)$ by saying that $\Ho$ takes $R(0101)$ to itself, but reversing the orientation of all its components. Indeed, for $i=1,\ldots,4$, we have that $\Ho$ takes the $i$-th component $R(0101)_i$ of $R(0101)$ to $(-1)\cdot R(0101)_i$. Thus $\Ho$ witnesses that $R(0101)$ admits the intrinsic symmetry $(1,-1,-1,-1,-1,(1\,2\,3\,4))$.
}
\label{fig:f650}
\end{figure}

If once again we take into account the orientations of the components, we may then conclude that $\Ho$ takes the link $R(0101)$ to itself, but reversing the orientations of all its components. For instance, $\Ho$ takes the first component $R(0101)_1$ of $R(0101)$ to itself {with its orientation reversed}. That is, $\Ho$ takes $R(0101)_1$ to $(-1)\cdot R(0101)_1$. Moreover, for $i=1,\ldots,4$ we have that $\Ho$ takes $R(0101)_i$ to $(-1)\cdot R(0101)_{i}$.

In general, suppose that that $L$ is a link whose components are given in some order $L_1,\ldots,L_n$, and let $L_1^*,\ldots,L_n^*$ be the respective components of the mirror image $L^*$ of $L$. (As we further discuss below, mirror images are irrelevant in our current context, but they are still an essential part of the definition of an intrinsic symmetry of a link). Let $(\epsilon_1,\ldots,\epsilon_n) \in \ZZ_2^n$, and let $\pi$ be a permutation of $[n]$. Then,

\begin{enumerate}

\item[(1)]   $L$ {\em admits} $(1,\epsilon_1,\ldots,\epsilon_n,\pi)$ if there is an ambient isotopy that maps $L$ to itself, taking $L_i$ to $\epsilon_i \cdot L_{\pi(i)}$ for $i=1,\ldots,n$;

\item[(2)]   $L$ {\em admits} $(-1,\epsilon_1,\ldots,\epsilon_n,\pi)$ if there is an ambient isotopy that maps $L$ to $\oL$, taking $L_i$ to $\epsilon_i \cdot L_{\pi(i)}^*$ for $i=1,\ldots,n$.

\end{enumerate}

If $L$ admits $(\epsilon_0,\epsilon_1,\ldots,\epsilon_n,\pi)$ for some $\epsilon_0\in\{-1,1\}$, then $(\epsilon_0,\epsilon_1,\ldots,\epsilon_n,\pi)$ is an {\em intrinsic symmetry} of $L$. For instance, Figure~\ref{fig:f650} illustrates that $R(0101)$ admits the intrinsic symmetry $(1,-1,-1,-1,-1,(1\,2\,3\,4))$. The set of intrinsic symmetries of a link $L$ forms a group, the {\em intrinsic symmetry group} of $L$~\cite{whitten}. If $L$ has $n$ components, then the identity element of this group is the trivial intrinsic symmetry $(1,{\underbrace{1,\, \ldots,\,1}_{n\,\text{\hglue 0.1 cm $1$s}}},(1\,2$ $\cdots\,n))$.

Once again, we include (2) (that is, we involve mirror images in the discussion) because it is an integral part of the notion of an intrinsic symmetry, but in our current context it is totally irrelevant: our interest lies on positive links, and the mirror image of a positive link is not a positive link. Throughout this paper we will not encounter any intrinsic symmetry of the form $(-1,\epsilon_1,\ldots,\epsilon_n,\pi)$.

\subsection{The intrinsic symmetry groups of small oscillating ring links}\label{sub:symsmallring}

Calculating the intrinsic symmetry group of a link $L$ is in general a very difficult task~\cite{sym4010143,canta1,livingston2021intrinsic}, but if $L$ is a reasonably small hyperbolic link (see~\cite{henryweeks}) then this group can be computed using {\tt SnapPy}~\cite{SnapPy}. One uses {\tt SnapPy} to calculate the full symmetry group (that is, the mapping class group of the pair $(S^3,L)$), and from this it is easy to obtain the intrinsic symmetry group of $L$.

The links $R(0101)$ and $R(1010)$ happen to be hyperbolic, and we followed the approach described in~\cite[Section 3]{canta1} to compute the intrinsic symmetry groups of these two links using {\tt SnapPy}.

\begin{fact}\label{fac:snappyring}
The intrinsic symmetry group of $R(0101)$ is isomorphic to $\ZZ_2\times \ZZ_2$. Its elements are $(1,1,1,1,1,\identity), (1,{-}1,-1,-1,-1,\identity)$, $(1,1,1,1,1, \reverse)$, and $(1,-1,-1,-1,-1, \reverse)$. The intrinsic symmetry group of $R(1010)$ is identical. 
\end{fact}

We emphasize that we knew that $R(0101)$ and $R(1010))$ admitted these four symmetries. Indeed, the first symmetry is trivial, and the second symmetry is witnessed by $\Ho$, as illustrated in Figure~\ref{fig:f650}. It is easy to see that the isotopy $\Ve$ illustrated in Figure~\ref{fig:f550} witnesses the third symmetry, and that the fourth symmetry is witnessed by the isotopy $\Ve\circ\Ho$.

\subsection{The base case of the proof of Lemma~\ref{lem:oscrings2}}\label{sub:smallring}

The next statement is the main result in this section, which corresponds to the base case of the proof of Lemma~\ref{lem:oscrings2}. As we shall see in the next section, even though this claim is stated in terms of oscillating links, and not in terms of oscillating sublinks (as Lemma~\ref{lem:oscrings2}), this statement is indeed equivalent to the case $n=4$ of that lemma.

\begin{claim}\label{cla:oscrings2}
Let $a=a_1 a_2 a_3 a_4$ be an oscillating word of length $4$. Let $b=b_1b_2b_3b_4$ be a word such that $R(a)\sim R(b)$, and let $\ii$ be an $\isot{R(a)}{R(b)}$ isotopy. Then,
\begin{enumerate}
\item[(1)] $R(b)$ is also oscillating, that is, $b$ is an oscillating word; and
\item[(2)] the $(R(a),R(b))$-permutation under $\ii$ is either the identity permutation $\identity$ or the reverse permutation $\reverse$.
\end{enumerate}
\end{claim}


\begin{proof}
We start by noting that (1) states that if $a\,{\in}\,\{0101,1010\}$, and $R(a){\sim}R(b)$, then $b\,{\in}\,\{0101,$ $1010\}$. We verified this using {\tt SageMath}~\cite{sagemath}: we found out that the Jones polynomials $V_{R(0101)}$ of $R(0101)$ and $V_{R(1010)}$ of $R(1010)$ are the same (this was of course expected, since these links are equivalent), and the Jones polynomial of any other ring link of size $4$ is distinct from $V_{R(0101)}$. 

To prove (2), suppose first that $b=a$. Let $\ii$ be an $\isot{R(a)}{R(a)}$ isotopy, and let $\pi$ be the $(R(a),R(a))$-permutation under $\ii$. Since $\ii$ maps each component of $R(a)$ to a component of $R(a)$ with its correct orientation, then $(1,1,1,1,1,\pi)$ must be an intrinsic symmetry of $R(a)$. Since $a$ is either $0101$ or $1010$, by Fact~\ref{fac:snappyring} we conclude that $\pi$ is either $\identity$ or $\reverse$. 

Suppose finally that $b=\Ov{a}$. Let $\ii$ be an $\isot{R(a)}{R(\Ov{a})}$ isotopy, and let $\pi$ be the $(R(a),R(\Ov{a}))$-permutation under $\ii$. Clearly, $R(\Ov{a})$ is the same as $R(a)$ but with the orientation of all its components reversed. Therefore $\ii$ maps each component of $R(a)$ to a component of $R(\Ov{a})$ with its orientation reversed, and so $(1,-1,-1,-1,-1,\pi)$ must be an intrinsic symmetry of $R(a)$. Since $a$ is either $0101$ or $1010$, by Fact~\ref{fac:snappyring}, we conclude that $\pi$ is either $\identity$ or $\reverse$.
\end{proof}

\section{Proof of Lemma~\ref{lem:oscrings2}}\label{sec:proofoscrings2}

The proof of Lemma~\ref{lem:oscrings2} relies on a natural equivalence (isotopy) between ring links and sublinks of ring links. An analogous equivalence between sublinks of boot links and boot links (respectively, between sublinks of flower links and flower links) will also play a central role in the proof of Theorem~\ref{thm:arrs} (respectively, Theorem~\ref{thm:arrf}).


\subsection{Sublinks of ring links are equivalent to ring links}\label{sub:isotopies}

Let $a=a_1\cdots a_n$ be a word, and let $a_{i_1} \cdots a_{i_k}$ be a subword of $a$. As we illustrate in Figure~\ref{fig:f210}, the special structure of the arrangement $\rr_n$ implies that the sublink $R(a)_{i_1,\ldots,i_k}$ of $R(a)$ is equivalent to the ring link $R(a_{i_1} \cdots a_{i_k})$, via a {strong} isotopy. Loosely speaking, one can ``bring together'' some components of $R(a)_{i_1,\ldots,i_k}$ until all the components are placed exactly in the same way as the components of $R(a_{i_1}\cdots a_{i_k})$. 

\def\Atg{{\tg{R(00\gl{0}1\gl{0}\gl{1})}}}
\def\tf#1{{\Scale[1.7]{#1}}}
\def\tg#1{{\Scale[1.8]{#1}}}
\def\th#1{{\Scale[3.2]{#1}}}
\def\tj#1{{\Scale[4.0]{#1}}}

\begin{figure}[ht!]
\centering
\scalebox{0.47}{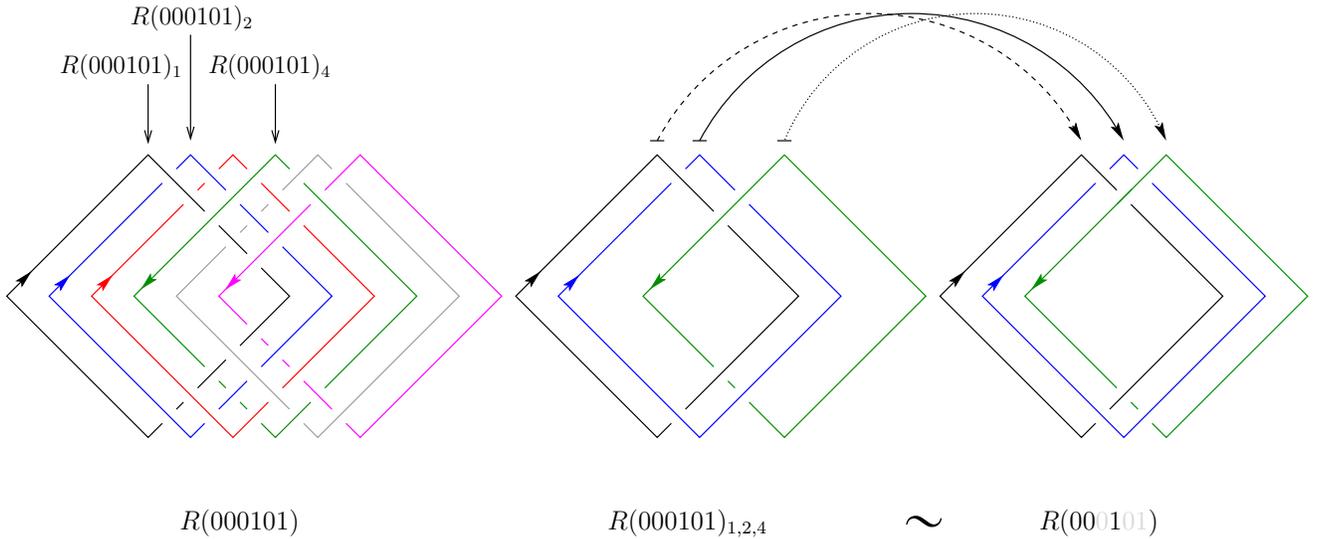}
\caption{Illustration of Observation~\ref{obs:subr}. Let $a=(a_1a_2a_3a_4a_5a_6)=(000101)$, and consider the subword $a_1a_2a_4=001$ of $a$. On the left-hand side we have $R(000101)$, and we indicate the components $R(000101)_1, R(000101)_2, R(000101)_4$ corresponding to $a_1=0, a_2=0$, and $a_4=1$, respectively. At the center we have the sublink $R(000101)_{1,2,4}$ that is the union of these three components. Loosely speaking, we can ``bring together'' some components of $R(000101)_{1,2,4}$ until they match exactly the components of the ring link $R(a_1a_2a_4)=R(001)$, illustrated at the right-hand side of the figure. This is a strong $\isot{R(000101)_{1,2,4}}{R(001)}$ isotopy. Indeed, its associated permutation is the identity permutation $\identity=(1\,2\,3)$, as the isotopy takes the $i$-th component of $R(000101)_{1,2,4}$ to the $i$-th component of $R(001)$ for $i=1,2,3$.}
\label{fig:f210}
\end{figure}

\def\so#1{{\Scale[1.8]{#1}}}
\def\sp#1{{\Scale[1.2]{#1}}}
\def\sq#1{{\Scale[1.3]{#1}}}

Such an $\isot{R(a)_{i_1,\ldots,i_k}}{R(a_{i_1}\cdots a_{i_k})}$ isotopy is indeed strong, as it clearly takes the $j$-th component $R(a)_{i_j}$ of $R(a)_{i_1,\ldots,i_k}$ to the $j$-th component $R(a_{i_1} \cdots a_{i_k})_j$ of the ring link $R(a_{i_1} \cdots a_{i_k})$, for $j=1,\ldots,k$. 

Clearly, we can reverse the process of ``bringing together'' the components of $R(a)_{i_1,\ldots,i_k}$ to the components of $R(a_{i_1}\cdots a_{i_k})$. Loosely speaking, we can ``pull apart'' some components of $R(a_{i_1}\cdots a_{i_k})$ until they are all placed exactly in the same way as the components of $R(a)_{{i_1},\ldots,i_k}$. Therefore there also exists a strong $\isot{R(a_{i_1}\cdots a_{i_k})}{R(a)_{i_1,\ldots,i_k}}$ isotopy. 

Let us highlight these crucial facts.

\begin{observation}\label{obs:subr}
Let $a=a_1\cdots a_n$ be a word. If $a_{i_1}\cdots a_{i_k}$ is any subword of $a$, then there exists an $\isop{R(a)_{i_1,\ldots,i_k}}{R(a_{i_1} \cdots a_{i_k})}{\identity}$ isotopy, and there exists an $\isop{R(a_{i_1}\cdots a_{i_k})}{R(a)_{i_1,\ldots,i_k}}{\identity}$ isotopy. 
\end{observation}

We close this section with a final observation that will be invoked repeatedly in the proof of Proposition~\ref{pro:arrr}.

\begin{observation}\label{obs:subr2}
Let $a=a_1\cdots a_n$ and $b=b_1\cdots b_n$ be words. Let $1\le i_1 < \cdots <i_k \le n$ and $1\le j_1 < \cdots < j_k \le n$ be integers, and let $\pi$ be a permutation of $[k]$. Then there exists an $\isop{R(a)_{i_1,\ldots,i_k}}{R(b)_{j_1,\ldots,j_k}}{\pi}$ isotopy if and only if there exists an $\isop{R(a_{i_1}\cdots a_{i_k})}{R(b_{j_1}\cdots b_{j_k})}{\pi}$ isotopy.
\end{observation}

\begin{proof}
We prove the ``if'' part. The proof of the ``only if'' part is totally analogous.

Suppose that there exists an $\isop{R(a_{i_1}\cdots a_{i_k})}{R(b_{j_1}\cdots b_{j_k})}{\pi}$ isotopy $\ii$. We know from Observation~\ref{obs:subr} there there exists an $\isop{R(a)_{i_1,\ldots,i_k}}{R(a_{i_1}\cdots a_{i_k})}{\identity}$-isotopy $\jj$, and the same observation implies that there exists an $\isop{R(b_{j_1}\cdots b_{j_k})}{R(b)_{j_1,\ldots,j_k}}{\identity}$ isotopy $\kk$.

Two applications of Observation~\ref{obs:comprings} yield that $\Isop{\kk\circ\ii\circ\jj}{R(a)_{i_1,\ldots,i_k}}{R(b)_{j_1,\ldots,j_k}}{\identity\circ\pi\circ\identity}$. Since $\identity\circ\pi\circ\identity=\pi$, this means that $\kk\circ\ii\circ\jj$ is an $\isop{R(a)_{i_1,\ldots,i_k}}{R(b)_{j_1,\ldots,j_k}}{\pi}$ isotopy.
\end{proof}





\subsection{Proof of Lemma~\ref{lem:oscrings2}}\label{sub:proofprooscrings2}

Even though in principle it is possible to prove Lemma~\ref{lem:oscrings2} in its given form, it happens to be easier to establish instead the following proposition, stated in terms of ring links instead of in terms of sublinks of ring links. We remark that the equivalence of this statement with Lemma~\ref{lem:oscrings2} is a consequence of Observation~\ref{obs:subr2}.

\begin{lemma}[Equivalent to Lemma~\ref{lem:oscrings2}]\label{lem:oscrings0}
Let $a=a_1 \cdots a_n$ be an oscillating word of length $n\ge 4$. Let $b$ be a word such that $R(a)\sim R(b)$, and let $\ii$ be an $\isot{R(a)}{R(b)}$ isotopy. Then,
\begin{enumerate}
\item[(1)] the $(R(a),R(b))$-permutation under $\ii$ is either the identity permutation $\identity$ or the reverse permutation $\reverse$; and
\item[(2)] $b$ is an oscillating word, that is, the ring link $R(b)$ is also oscillating.
\end{enumerate}
\end{lemma}

\begin{proof}
We proceed by induction on the length $n$ of $a$. Claim~\ref{cla:oscrings2} shows that the statement is true for $n{=}4$. For the inductive step we let $m\ge 4$ be an integer, assume that the lemma holds for oscillating words of length $m$, and prove that then it holds for an oscillating word of length $m+1$. 

Thus we let $a=a_1\cdots a_m a_{m+1}$ be an oscillating word, let $b=b_1 \cdots b_{m+1}$ be a word such that $R(a)\sim R(b)$. Let $\ii$ be an $\isot{R(a)}{R(b)}$ isotopy, and let $\pi$ be the $(R(a),R(b))$-permutation under $\ii$. Our goal is to show that (I) $\pi$ is either $\identity$ of $\reverse$; and (II) $b$ is oscillating. 

Let $R(b)_{j_1,\ldots,j_{m}}$ (respectively, $R(b)_{k_1,\ldots,k_{m}}$) be the sublink of $R(b)$ that is the image of $R(a)_{1,\ldots,m}$ (respectively, $R(a)_{2,\ldots,m+1}$) under $\ii$. 

By Observation~\ref{obs:subr2}, the induction hypothesis implies that the $({R(a)_{1,\ldots,m}},{R(b)_{j_1,\ldots,j_m}})$-permuta\-tion under ${\ii}$ is either $\identity$ or $\reverse$. That is, either 
$$\hbox{(i) $\Isop{\ii}{R(a)_{1,\ldots,m}}{R(b)_{j_1,\ldots,j_m}}{\identity}$ or (ii) $\Isop{\ii}{R(a)_{1,\ldots,m}}{R(b)_{j_1,\ldots,j_m}}{\reverse}$.}$$ 
Similarly, the $({R(a)_{2,\ldots,m+1}},{R(b)_{k_1,\ldots,k_m}})$-permutation under ${\ii}$ is either $\identity$ or $\reverse$. Therefore either 
$$\hbox{(iii) $\Isop{\ii}{R(a)_{2,\ldots,m+1}}{R(b)_{k_1,\ldots,k_m}}{\identity}$ or (iv) $\Isop{\ii}{R(a)_{2,\ldots,m+1}}{R(b)_{k_1,\ldots,k_m}}{\reverse}$.}$$
Note that (i) implies that $\pi(1) < \cdots < \pi(m)$, (ii) implies that $\pi(m) > \cdots {>}\pi(1)$, (iii) implies that $\pi(2) < \cdots {<}\pi(m+1)$, and (iv) implies that $\pi(m+1) > \cdots {>}\pi(2)$. Evidently (i) and (iv) are inconsistent with each other, and (ii) and (iii) are inconsistent with each other. Therefore either (i) and (iii) hold, or (ii) and (iv) hold. 

That is, either $\pi(1) < \pi(2) <\cdots {<}\pi(m+1)$ or $\pi(m+1) > \cdots {>}\pi(2) > \pi(1)$. In the former case necessarily $\pi(\ell)=\ell$ for $i=1,\ldots,m+1$, and so $\pi=\identity$, and in the former case $\pi(\ell)=(m+1)-\ell+1$ for $i=1,\ldots,m+1$, that is, $\pi=\reverse$. Thus (I) holds.

We prove (II) under the assumption that the $(R(a),R(b))$-permutation under $\ii$ is $\reverse$. The proof for the case in which this permutation is $\identity$ is totally analogous.

Under this assumption 
$$\hbox{($\dag$) $\Isot{\ii}{R(a)_{1,\ldots,m}}{R(b)_{2,\ldots,m+1}}$ and ($\ddag$)  $\Isot{\ii}{R(a)_{2,\ldots,m+1}}{R(b)_{1,\ldots,m}}$.}$$ 
By Observation~\ref{obs:subr2}, the induction hypothesis applied to ($\dag$) implies that $R(b)_{2,\ldots,m+1}$ is oscillating, that is, 
$$\hbox{($*$) {\em $b_{2} \cdots b_{m+1}$ is an oscillating subword of $b$.}}$$ 
Similarly, Observation~\ref{obs:subr2} and the induction hypothesis applied to ($\ddag$) imply that $R(b)_{1,\ldots,m}$ is oscillating, that is, 
$$\hbox{($**$) {\em $b_1\cdots b_{m}$ is an oscillating subword of $b$.}}$$
Clearly, ($*$) and ($**$) imply that $b$ is oscillating.
\end{proof}

\section{ {Boot links:} proof of Theorem~\ref{thm:arrs}}\label{sec:proofarrs}

Throughout this paper we refer to a link in $\linksor{\bb_n}$ as a {\em positive boot link of size $n$} or simply as a {\em boot link of size $n$}, since all links under consideration are positive.

As we mentioned in Section~\ref{sec:intro}, the proof of Theorem~\ref{thm:arrs} follows a similar strategy to the one we used to prove Theorem~\ref{thm:arrr}. This is actually an understatement: with the framework we developed for the proof of Theorem~\ref{thm:arrr}, the proof of Theorem~\ref{thm:arrs} will be much shorter, as we only need to describe some adjustments we need to make in order to deal with boot links.

\subsection{Correspondence between boot links and binary words}

Similarly as with ring links, a boot link of size $n$ is naturally associated to a binary word of size $n$. Each link in $\linksor{\bb_n}$ is obtained by assigning an orientation to each of the $n$ pseudocircles in $\bb_n$, and an orientation assignment can be naturally encoded in a totally analogous manner as we did for ring links. We illustrate this in Figure~\ref{fig:f120}.

\def\so#1{{\Scale[1.4]{#1}}}
\def\sp#1{{\Scale[1.15]{#1}}}

\begin{figure}[ht!]
\centering
\scalebox{0.6}{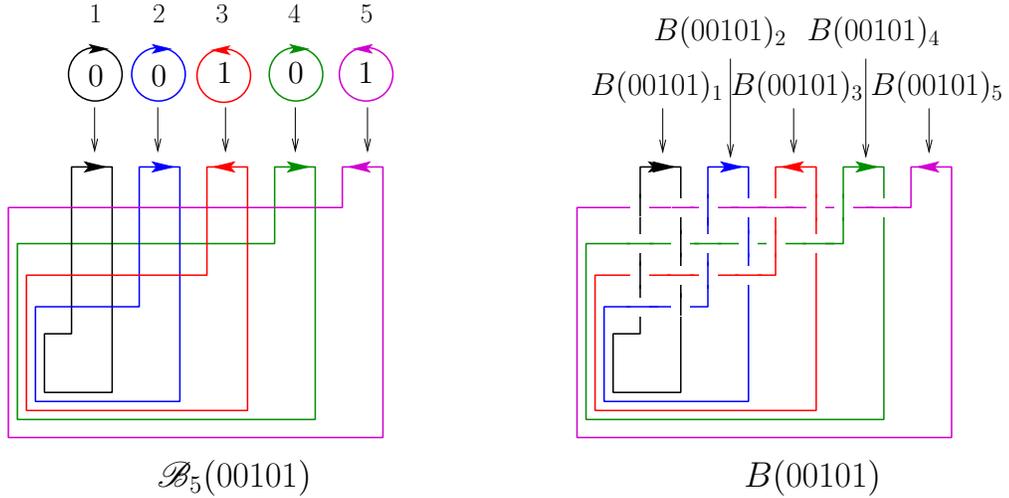}
\caption{The oriented arrangement $\bb_5(00101)$ and its induced positive link, the boot link $B(00101)$. We also indicate the five components of $B(00101)$.}
\label{fig:f120}
\end{figure}

Formally, let $i=1,2,\ldots,n$ be the pseudocircles in $\bb_n$, ordered from left to right. As with ring arrangements, we encode an orientation using a word $a$ of length $n$, where the $i$-th entry of the word $a$ is $0$ (respectively, $1$) if pseudocircle $i$ is oriented clockwise (respectively, counterclockwise). Given such a word $a$, we use $\bb_n(a)$ to denote the corresponding oriented arrangement. Finally, we use $B(a)$ to denote the positive link {induced} by $\bb_n(a)$ (here $B$ stands for ``boot'').

Thus, as with ring links, {\em there is a one-to-one correspondence between the set of all binary words of length $n$ and the collection $\linksor{\bb_n}$ of all boot links of size $n$}.

\subsection{Reducing Theorem~\ref{thm:arrs} to a proposition}

Back in Section~\ref{sec:proofthmarrr} we reduced Theorem~\ref{thm:arrr} to Proposition~\ref{pro:arrr}. Here we proceed similarly, but the situation is simpler for boot links. 

In order to prove Theorem~\ref{thm:arrr} we identified some natural isotopies on ring links, culminating with Observation~\ref{obs:ifring}, which gave sufficient conditions for a ring link to be equivalent to a ring link: if $b$ is either $a,\Ov{a},\In{a}$, or $\IO{a}$, then $R(a)\sim R(b)$.

For boot links we do not need to go through that step: the equivalent to Observation~\ref{obs:ifring} for boot links would be the trivial statement that 
$$\hbox{($*$) {\em if $b=a$, then $B(a)=B(b)$}.}$$
The workhorse behind the proof of Theorem~\ref{thm:arrs} is then the following proposition, which claims that the converse of ($*$) holds, as long as the rank of $a$ is at least $6$.

\begin{proposition}[Implies Theorem~\ref{thm:arrs}]\label{pro:arrs}
Let $a$ be a word of rank {$r\ge 6$}. If $b$ is a word such that $B(a)\sim B(b)$, then $b=a$.
\end{proposition}

As we did with ring links, we defer the proof of the proposition for the moment, and show that Theorem~\ref{thm:arrs} easily follows from it.

\begin{proof}[Proof of Theorem~\ref{thm:arrs} (assuming Proposition~\ref{pro:arrs})]
Let $a$ be a word of length $n$. Using standard calculations one obtains that the probability that $\rank{a}$ is less than $6$ goes to $0$ as $n\to \infty$. By Proposition~\ref{pro:arrs}, this implies that the probability that $B(a)$ is not equivalent to any other boot link goes to $1$ as $n\to\infty$.

The one-to-one correspondence between binary words of length $n$ and elements of $\linksor{\bb_n}$ then implies that the probability that the equivalence class of a random link in $\linksor{\bb_n}$ has size $1$ goes to $1$ as $n \to \infty$. Since $|\linksor{\bb_n}|=2^n$, Theorem~\ref{thm:arrs} follows. 
\end{proof}


\section{Proof of Proposition~\ref{pro:arrs}}\label{sub:proofboots}

Similarly as in the proof of Proposition~\ref{pro:arrr}, sublinks of boot links play a central role in the proof of Proposition~\ref{pro:arrs}, so we start with a brief discussion on them.

\subsection{Sublinks of boot links are equivalent to boot links}

Let $a=a_1\ldots a_n$ be a  word, and let $i_1,\ldots,i_k$ be integers such that $1 \le i_1 < \cdots < i_k \le n$. Then $a_{i_1}\cdots a_{i_k}$ is a subword of $a$, and this subword naturally corresponds to a sublink $B(a)_{i_1}\cup \cdots \cup B(a)_{i_k}$ of $B(a)$. For brevity, we use $B(a)_{i_1,\ldots,i_k}$ to denote this link. 

We say that the boot link $B(a)$ is {\em oscillating} if the word $a$ is oscillating, and we say that the sublink $B(a)_{i_1,\ldots,i_k}$ of $B(a)$ is {\em oscillating} if $a_{i_1} \cdots a_{i_k}$ is an oscillating subword of $a$. No oscillating sublink of $B(a)$ can have size larger than the rank $r$ of $a$, since this is the length of a longest oscillating subword of $a$. 

Similarly as in the case of ring links, as we illustrate in Figure~\ref{fig:f380} the special structure of the arrangement $\bb_n$ implies that the sublink $B(a)_{i_1,\ldots,i_k}$ of $B(a)$ is equivalent to the boot link $B(a_{i_1} \cdots a_{i_k})$, via a {strong} isotopy. Loosely speaking, one can ``bring together'' some components of $B(a)_{i_1,\ldots,i_k}$ until all the components are placed exactly in the same way as the components of $B(a_{i_1}\cdots a_{i_k})$, and we can also reverse this process to take the components of $B(a_{i_1}\cdots a_{i_k})$ to $B(a)_{i_1,\ldots,i_k}$. 

\begin{observation}\label{obs:subs}
Let $a=a_1\cdots a_n$ be a word. If $a_{i_1}\cdots a_{i_k}$ is any subword of $a$, then there exists an $\isop{B(a)_{i_1,\ldots,i_k}}{B(a_{i_1} \cdots a_{i_k})}{\identity}$ isotopy, and there exists an $\isop{B(a_{i_1}\cdots a_{i_k})}{B(a)_{i_1,\ldots,i_k}}{\identity}$ isotopy. 
\end{observation}

We thus obtain the following crucial statement, which parallels Observation~\ref{obs:subr2}.

\begin{observation}\label{obs:subs2}
Let $a=a_1\cdots a_n$ and $b=b_1\cdots b_n$ be words. Let $1\le i_1 < \cdots <i_k \le n$ and $1\le j_1 < \cdots < j_k \le n$ be integers, and let $\pi$ be a permutation of $[k]$. Then there exists an $\isop{B(a)_{i_1,\ldots,i_k}}{B(b)_{j_1,\ldots,j_k}}{\pi}$ isotopy if and only if there exists an $\isop{B(a_{i_1}\cdots a_{i_k})}{B(b_{j_1}\cdots b_{j_k})}{\pi}$ isotopy.
\end{observation}

\def\Btg{{\tg{B(00\gl{0}1\gl{0}\gl{1})}}}
\def\tf#1{{\Scale[1.7]{#1}}}
\def\tg#1{{\Scale[1.8]{#1}}}
\def\th#1{{\Scale[3.2]{#1}}}
\def\tj#1{{\Scale[4.0]{#1}}}

\begin{figure}[ht!]
\centering
\scalebox{0.6}{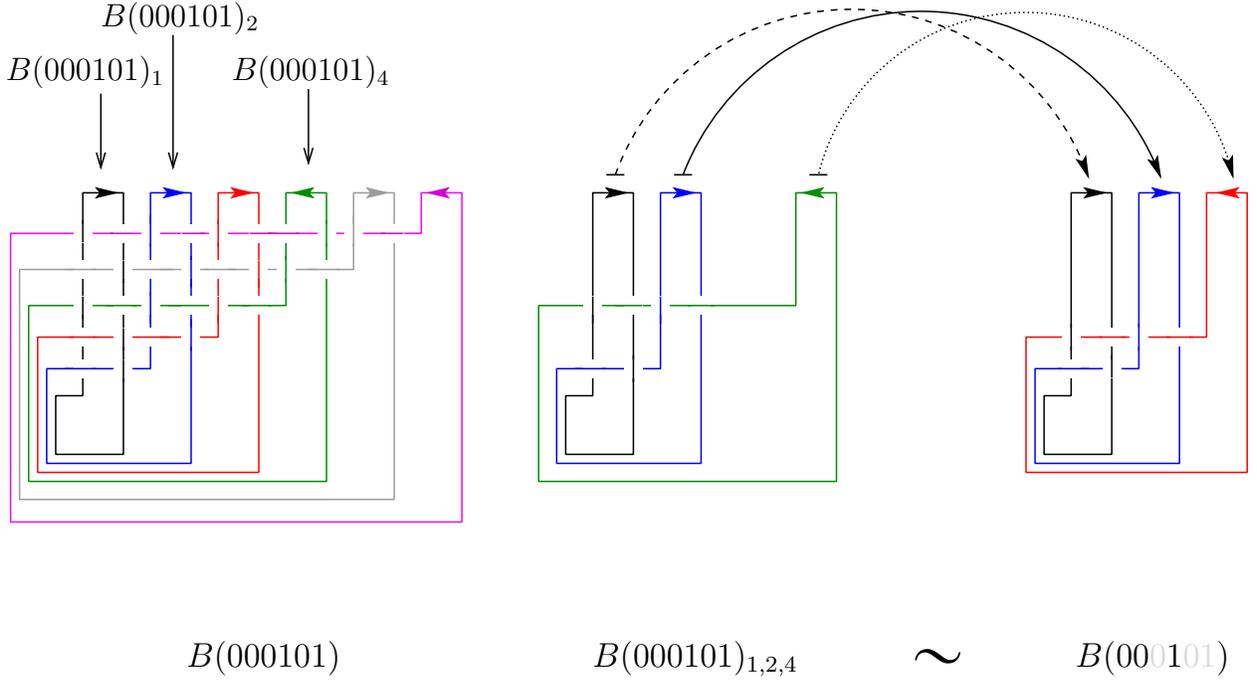}
\caption{Illustration of Observation~\ref{obs:subs}.} 
\label{fig:f380}
\end{figure}

\def\so#1{{\Scale[1.8]{#1}}}
\def\sp#1{{\Scale[1.2]{#1}}}
\def\sq#1{{\Scale[1.3]{#1}}}


\subsection{Reducing Proposition~\ref{pro:arrs} to a lemma}

Back in Section~\ref{sec:proofproparrr} we reduced Proposition~\ref{pro:arrr} to Lemmas~\ref{lem:workhorsering} and~\ref{lem:oscrings2}, and these lemmas were proved in later sections. Here we proceed in a similar way. The next statements parallel these lemmas. 

As we shall see, in this case we will only need to defer the proof of Lemma~\ref{lem:oscboots2} to the next section,  as the proof of Lemma~\ref{lem:workhorseboot} is virtually identical to the proof of Lemma~\ref{lem:workhorsering}.

\begin{lemma}\label{lem:workhorseboot}
Let $a$ be a word with rank $r\ge 6$, let $b$ be a word such that $B(a)\sim B(b)$, and let $\ii$ be an $\isot{B(a)}{B(b)}$ isotopy. Let $B(a)_{i_1,\ldots,i_r}$ be an oscillating sublink of $B(a)$, and let $B(b)_{j_1,\ldots,j_r}$ be its image under $\ii$. Then $b$ is the $\pi$-image of $a$, where $\pi$ is the $(B(a_{i_1,\ldots,i_r}),B(b_{j_1,\ldots,j_r}))$-permutation under $\ii$.
\end{lemma}

\begin{lemma}\label{lem:oscboots2}
Let $a=a_1 \cdots a_n$ be a word of rank $r\ge 6$, and let $B(a)_{i_1,\ldots,i_r}$ be an oscillating sublink of $B(a)$ of size $r$. Let $b$ be a word such that $B(a)\sim B(b)$, and let $\ii$ be an $(B(a),B(b))$ isotopy. Let $B(b)_{j_1,\ldots,j_r}$ be the sublink of $B(b)$ that is the image of $B(a)_{i_1,\ldots,i_r}$ under $\ii$. Then,
\begin{enumerate}
\item[(1)] the $(B(a)_{i_1,\ldots,i_r},B(b)_{j_1,\ldots,j_r})$-permutation under $\ii$ is the identity permutation $\identity$; and
\item[(2)] $b_{j_1}\cdots b_{j_r}=a_{i_1}\cdots a_{i_r}$. 
\end{enumerate}
\end{lemma}

We have the following easy consequence of Lemma~\ref{lem:oscboots2}, which is the analogue of Corollary~\ref{cor:ringsrank}.

\begin{corollary}\label{cor:bootsrank}
Let $a=a_1 \cdots a_n$ be a word of rank $r\ge 6$, and suppose that $b$ is a word such that $B(a)\sim B(b)$. Then $\rank{b}=r$.
\end{corollary}

\begin{proof}[Proof of Lemma~\ref{lem:workhorseboot} (assuming Lemma~\ref{lem:oscboots2})]
This is virtually identical to the proof of Lemma~\ref{lem:workhorsering}. It suffices to invoke Corollary~\ref{cor:bootsrank} instead of Corollary~\ref{cor:ringsrank}, and to invoke Lemma~\ref{lem:oscboots2}(2) instead of Lemma~\ref{lem:oscrings2}(2).
\end{proof}

\begin{proof}[Proof of Proposition~\ref{pro:arrs} (assuming Lemma~\ref{lem:oscboots2})]
Let $a=a_1\cdots a_n$ be a word with rank $r\ge 6$. Let $b=b_1\cdots b_n$ be a word such that $B(a)\sim B(b)$, and let $\ii$ be an $(B(a),B(b))$ isotopy. Let $B(a)_{i_1,\ldots,i_r}$ be an oscillating sublink of $B(a)$, and let $B(b)_{j_1,\ldots,j_r}$ be the sublink of $B(b)$ that is the image of $B(a)_{i_1,\ldots,i_r}$ under $\ii$, and let $\pi$ be the $(B(a)_{i_1,\ldots,i_r},B(b)_{j_1,\ldots,j_r})$-permutation under $\ii$.

Let $a=A_1 \cdots A_r=a_{i_1}^{|A_1|}\cdots a_{i_r}^{|A_r|}$ be the canonical decomposition of $a$. We note that Corollary~\ref{cor:bootsrank} implies that the rank of $b$ is also $r$, and so we can let  $b=B_1 \cdots B_r=b_{j_1}^{|B_1|}\cdots b_{j_r}^{|B_r|}$ be the canonical decomposition of $b$. 

Lemma~\ref{lem:workhorseboot} implies that $b$ is the $\pi$-image of $a$, and Lemma~\ref{lem:oscboots2}(1) implies that $\pi$ is $\identity$. Therefore $|B_k|=|A_{\identity(k)}|=|A_k|$ for $k=1,\ldots,r$. By Lemma~\ref{lem:oscrings2}(2), we have that $b_{j_1}\cdots b_{j_r}=a_{i_1}\cdots a_{i_r}$, and so a glance at the canonical decompositions of $a$ and $b$ reveals that $b=a$.
\end{proof}

\section{Proof of Lemma~\ref{lem:oscboots2}}\label{sec:lemboots}

As the proof of Lemma~\ref{lem:oscrings2}, the proof of Lemma~\ref{lem:oscboots2} is by induction on $n$. In order to establish the base case we need the intrinsic symmetry groups of the two oscillating boot links of size $6$. Fortunately, these links are hyperbolic, and so we could use {\tt SnapPy} to calculate these groups. We obtained the following

\begin{fact}\label{fac:snappyboot}
The intrinsic symmetry groups of the boot links $B(010101)$ and $B(101010)$ are trivial. Therefore, if $\ii$ is an $\isot{B(010101)}{B(010101)}$ isotopy, then the $(B(010101),B(010101))$-permutation under $\ii$ is the identity permutation $\identity$. A totally analogous statements holds for $B(101010)$.
\end{fact}

We have the following analogue of Claim~\ref{cla:oscrings2}. This will be the base case of the inductive proof of Lemma~\ref{lem:oscboots2} (or, more precisely, for the proof of its equivalent Lemma~\ref{lem:oscboots0}).

\begin{claim}\label{cla:oscboots2}
Let $a=a_1 a_2 a_3 a_4 a_5 a_6$ be an oscillating word of length $6$. Let $b=b_1b_2b_3b_4b_5b_6$ be a word such that $B(a)\sim B(b)$, and let $\ii$ be an $\isot{B(a)}{B(b)}$ isotopy. Then,
\begin{enumerate}
\item[(1)] $b=a$; and 
\item[(2)] the $(B(a),B(b))$-permutation under $\ii$ is the identity permutation $\identity$.
\end{enumerate}
\end{claim}

\begin{proof}
We start by noting that (1) states that if $a=010101$ (respectively, $a=101010$) and $B(a){\sim}B(b)$, then $b=010101$ (respectively, $b=101010$). We verified this using {\tt SageMath}: the Jones polynomial of $V_{B(010101)}$ of $B(010101)$ is different from the Jones polynomial of any other boot link of size $6$, and the same holds for the Jones polynomial $V_{B(101010)}$ of $B(101010)$.

To prove (2), let $\ii$ be an $\isot{B(a)}{B(b)}$ isotopy, and let $\pi$ be the $(B(a),B(b))$-permutation under $\ii$. We already know that $b=a$, and so $\ii$ is an $\isot{B(a)}{B(a)}$ isotopy, and $\pi$ is the $(B(a),B(a))$-permutation under $\ii$. Thus $\ii$ maps each component of $B(a)$ to a component of itself with its given orientation, and so $(1,1,1,1,1,1,1,\pi)$ must be an intrinsic symmetry of $B(a)$. Regardless of whether $a=010101$ or $101010$, Fact~\ref{fac:snappyboot} implies that $\pi$ is the identity permutation $\identity$. 
\end{proof}


Similarly as we proceeded in the proof of Lemma~\ref{lem:oscrings2}, it turns out to be easier and more natural to establish instead the following lemma, stated in terms of boot links instead of in terms of sublinks of boot links. The equivalence of this statement with Lemma~\ref{lem:oscboots2} follows immediately from Observation~\ref{obs:subs2}.

\begin{lemma}[Equivalent to Lemma~\ref{lem:oscboots2}]\label{lem:oscboots0}
Let $a=a_1 \cdots a_n$ be an oscillating word of length $n\ge 6$. Let $b$ be a word such that $B(a)\sim B(b)$, and let $\ii$ be an $\isot{B(a)}{B(b)}$ isotopy. Then,
\begin{enumerate}
\item[(1)] the $(B(a),B(b))$-permutation under $\ii$ is the identity permutation $\identity$; and
\item[(2)] $b=a$.
\end{enumerate}
\end{lemma}

\begin{proof}
The proof is an easier version of the proof of Lemma~\ref{lem:oscrings0}. We proceed by induction on the length $n$ of $a$. Claim~\ref{cla:oscboots2} shows that the statement is true for $n{=}6$. For the inductive step we let $m\ge 6$ be an integer, assume that the lemma holds for oscillating words of length $m$, and prove that then it holds for an oscillating word of length $m+1$. 

Thus we let $a=a_1\cdots a_m a_{m+1}$ be an oscillating word, and let $b=b_1 \cdots b_{m+1}$ be a word such that $B(a)\sim B(b)$. Let $\ii$ be an $\isot{B(a)}{B(b)}$ isotopy, and let $\pi$ be the $(B(a),B(b))$-permutation under $\ii$. We must show that 
$$\hbox{$\pi=\identity$ and $b=a$.}$$

Let $B(b)_{j_1,\ldots,j_{m}}$ (respectively, $B(b)_{k_1,\ldots,k_{m}}$) be the sublink of $B(b)$ that is the image of $B(a)_{1,\ldots,m}$ (respectively, $B(a)_{2,\ldots,m+1}$) under $\ii$. 

By Observation~\ref{obs:subs2}, the induction hypothesis implies that the $({B(a)_{1,\ldots,m}},{B(b)_{j_1,\ldots,j_m}})$-permuta\-tion under ${\ii}$ is $\identity$. This implies in particular that $\pi(1) < \cdots < \pi(m)$. Similarly, the induction hypothesis implies that the $({B(a)_{2,\ldots,m+1}},{B(b)_{k_1,\ldots,k_m}})$-permutation under ${\ii}$ is $\identity$. This implies in particular that $\pi(2) < \cdots {<}\pi(m+1)$. Therefore $\pi(1) < \pi(2) <\cdots {<}\pi(m+1)$, and so necessarily $\pi(\ell)=\ell$ for $\ell=1,\ldots,m+1$. In other words, $\pi$ is the identity permutation $\identity$. Thus (I) holds.

To prove (2), note that we have in particular we have proved that 
$$\hbox{($\dag$) $\Isot{\ii}{B(a)_{1,\ldots,m}}{B(b)_{1,\ldots,m}}$ and ($\ddag$)  $\Isot{\ii}{B(a)_{2,\ldots,m+1}}{B(b)_{2,\ldots,m+1}}$.}$$
By Observation~\ref{obs:subr2}, the induction hypothesis applied to ($\dag$) implies that $b_1\cdots b_m=a_1 \cdots a_m$. Similarly, Observation~\ref{obs:subr2} and the induction hypothesis applied to ($\ddag$) imply that $b_2\cdots b_{m+1}=a_2\cdots a_{m+1}$. Evidently, (I) and (II) imply that $b_1\cdots b_{m+1}=a_1\cdots a_{m+1}$, that is, $b=a$.
\end{proof}




\section{{Flower links:} proof of Theorem~\ref{thm:arrf}}\label{sec:proofarrff}

We shall refer to a link in $\linksor{\ff_n}$ as a {\em positive flower link of size $n$} or simply as a {\em flower link of size $n$}, since all links under consideration are positive. We prove Theorem~\ref{thm:arrf} using the same strategy used to prove Theorems~\ref{thm:arrr} and~\ref{thm:arrs}. 

For convenience, we will deal exclusively with flower links of {\rm even} size. We could in principle also consider flower links of odd size, but this would imply some additional case analyses in several important statements, and so it turns out to be easier to focus on flower links of even size. We emphasize that this has no impact in the validity of Theorem~\ref{thm:arrf}, as it is easy to see that (due to its asymptotic character) if the theorem holds for even values of $n$, then it also holds when $n$ is odd {(see Proposition \ref{pro:arrf} and Corollary \ref{pro:arrf})}. 

\vglue 0.3 cm
\noindent{\bf Remark. }{\sl Whenever we work with a flower link, it will be implicitly assumed that it has an {\rm even} number of components.}
\vglue 0.3 cm

\subsection{Correspondence between flower links and binary words}

Similarly as with ring links and with boot links, we can naturally associate a flower link of size $n$ to a binary word of size $n$. Each link in $\linksor{\ff_n}$ is obtained by assigning an orientation to each of the $n$ pseudocircles in $\ff_n$, and an orientation assignment can be naturally encoded similarly as we did for ring links and for boot links. 

We note that in this case we need to come up with a convention: the components of a ring arrangement and the components of a boot arrangement are naturally ordered from left to right, but the components of a flower arrangement appear in a cyclic fashion. We refer the reader to Figure~\ref{fig:f480} to illustrate this discussion. A flower arrangement with an even number $n$ pseudocircles has two topmost pseudocircles. Thus we can label (order) the pseudocircles $1,\ldots,n$ starting with its {\em topmost right} component and following a clockwise order, as shown for the arrangement at the top of Figure~\ref{fig:f480}. 

\def\so#1{{\Scale[1.4]{#1}}}
\def\sp#1{{\Scale[1.15]{#1}}}
\def\sr#1{{\Scale[1.2]{#1}}}
\def\st#1{{\Scale[0.9]{#1}}}
\def\tf#1{{\Scale[1.4]{#1}}}

\begin{figure}[ht!]
\centering
\scalebox{0.8}{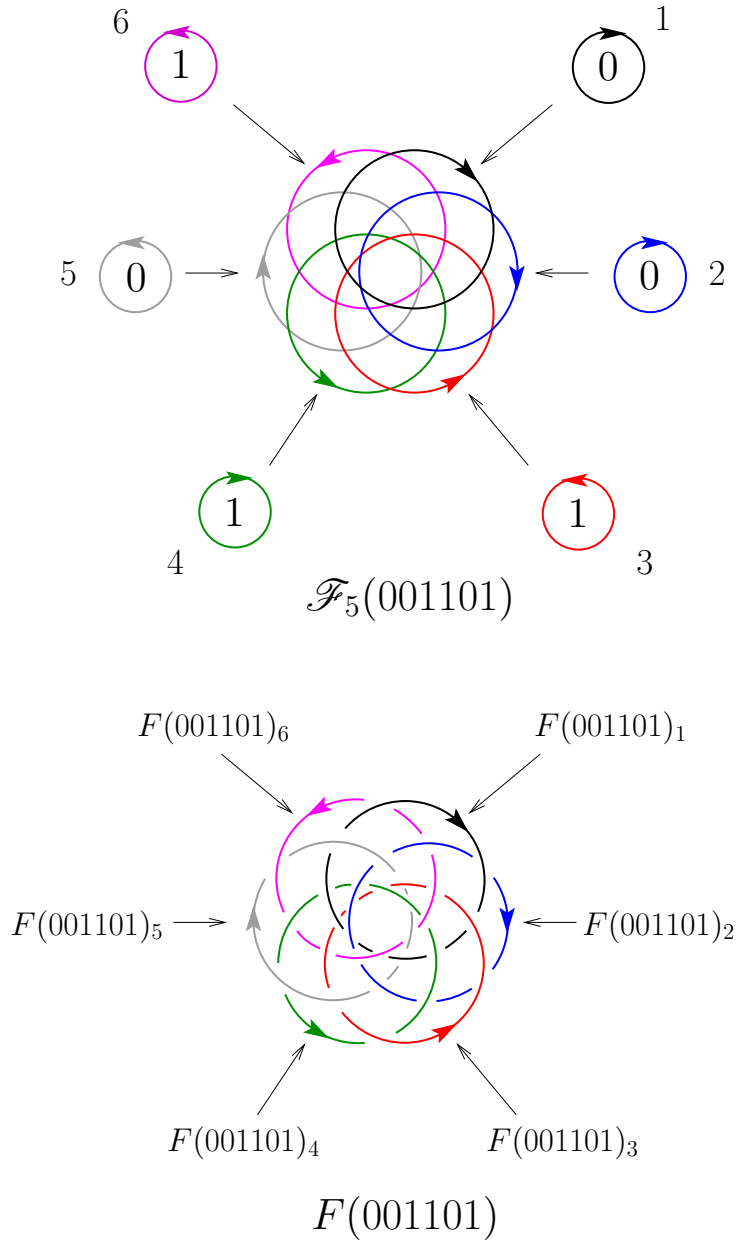}
\caption{The oriented arrangement $\ff_5(001101)$ and its induced link, the flower link $F(001101)$. We indicate the six components of $F(001101)$.}
\label{fig:f480}
\end{figure}

Formally, let $i=1,2,\ldots,n$ be the pseudocircles in $\ff_n$, ordered in a naturally clockwise cyclic fashion, starting at the topmost right pseudocircle. As with ring arrangements and boot arrangements, we encode an orientation using a word $a$ of length $n$, where the $i$-th entry of the word $a$ is $0$ (respectively, $1$) if pseudocircle $i$ is oriented clockwise (respectively, counterclockwise). Given such a word $a$, we use $\ff_n(a)$ to denote the corresponding oriented arrangement. Finally, we use $F(a)$ to denote the positive link {induced} by $\ff_n(a)$ (here $F$ stands for ``flower''). We refer the reader again to Figure~\ref{fig:f480} for an illustration. As we also illustrate in that figure, for each $i=1,\ldots,n$ we use $F(a)_i$ to denote the $i$-th component of the flower link $F(a)$.

Thus, as with ring links and boot links, {\em there is a one-to-one correspondence between the set of all binary words of length $n$ and the collection $\linksor{\ff_n}$ of all flower links of size $n$}.

\subsection{Isotopies that act naturally on flower links}\label{sub:isotopiesflower}

We recall that, at a high level, the proof of Theorem~\ref{thm:arrr} had two main steps. Given a ring link $R(a)$, the easy step was to exhibit isotopies that proved that $R(a)\sim R(\Ov{a}), R(a)\sim R(\In{a})$, and $R(a)\sim R(\IO{a})$ (evidently, there is no need to exhibit an isotopy that takes $R(a)$ to itself). This part culminated with Observation~\ref{obs:ifring}, and was by far the easier step of the proof: {\em if $b$ is either $a,\Ov{a},\In{a}$, or $\IO{a}$, then $R(a)\sim R(b)$}. The difficult part was the converse statement, namely Proposition~\ref{pro:arrr}: {\em if $R(a)\sim R(b)$, then $b$ is either $a,\Ov{a},\In{a}$, or $\IO{a}$}.

Here the strategy is totally analogous. Given a flower link $F(a)$, we start by exhibiting isotopies that prove that $F(a)$ is equivalent to a certain collection of flower links. This is the goal of the current subsection, whose main result is Observation~\ref{obs:ifflower}. The more difficult part will be stated in Proposition~\ref{pro:arrf}, which plays the role that Proposition~\ref{pro:arrr} played for ring links.

\subsubsection{The isotopy $\Ro_{2\pi/n}$, the word mapping $\ro$, and the permutation $\rot{n}$}

Let us then exhibit some natural isotopies that take a flower link to a flower link. As we illustrate in Figure~\ref{fig:f490} for the case $n=6$, we use $\Ro_{2\pi/n}$ to denote an isotopy that clockwise rotates an angle of $360^{\circ}/n$ a diagram along an axis perpendicular to the {sheet}. (We mostly use degrees in our discussions, but we find it more natural to use $2\pi/n$ as a subscript of $\Ro$, rather than $360^{\circ}/n$). 

\def\sz#1{{\Scale[1.25]{#1}}}
\def\Ctg{{\sz{(360^{\circ}/6=60^{\circ} \text{\rm \hglue 0.15cm in this example})}}}
\def\Dtg{{\sz{(\Ro_{2\pi/6} \text{\rm \hglue 0.15cm in this example})}}}
\def\th#1{{\Scale[2.2]{#1}}}

\begin{figure}[ht!]
\centering
\scalebox{0.8}{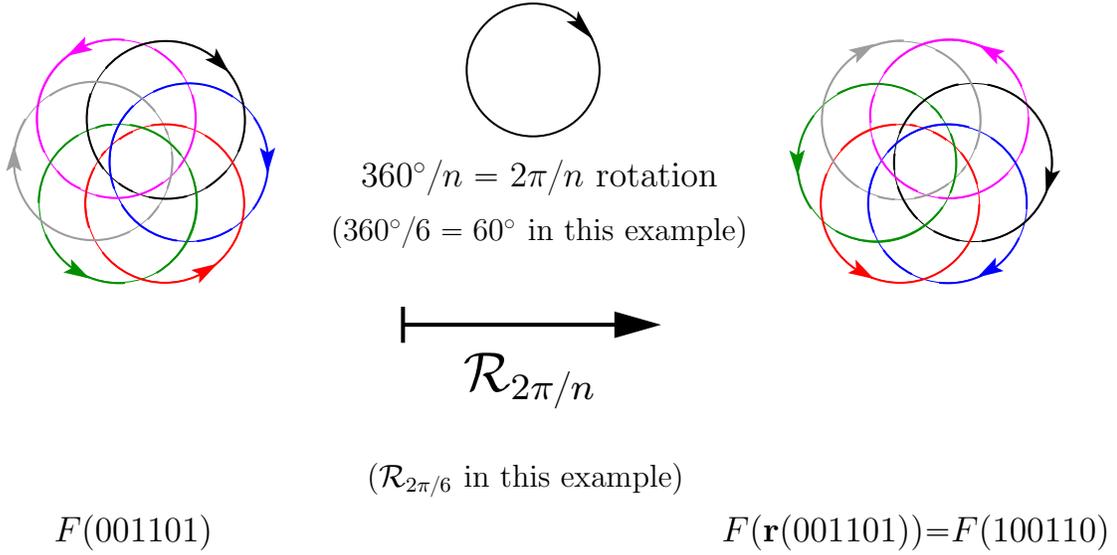}
\caption{For each positive integer $n$, the isotopy $\Ro_{2\pi/n}$ clockwise rotates an angle of $2\pi/n$ a diagram along an axis perpendicular to the page. This isotopy takes a flower link $F(a_1 a_2\cdots a_{n-1} a_n)$ to the flower link $F(\ro(a_1 a_2 \cdots a_{n-1} a_n))=F(a_n a_1 a_2 \cdots a_{n-1})$, and the $(F(a_1 a_2\cdots a_{n-1} a_n),F(a_n a_1 a_2 \cdots a_{n-1}))$-permutation under $\Ro_{2\pi/n}$ is $\rot{n}=(2\,3\,\cdots\,n\,1)$.  Here we illustrate $\Ro_{2\pi/6}$, which takes $F(001101)$ to $F(100110)$. The $(F(001101),F(100110))$-permutation under $\Ro_{2\pi/6}$ is indeed $\rot{6}=(2\,3\,4\,5\,6\,1)$, as it takes the $i$-th component of $F(001101)$ to the $(i\opi{6} 1)$-st component of $F(100110)$, for $i=1,\ldots,6$.}
\label{fig:f490}
\end{figure}

As we show in Figure~\ref{fig:f490}, if apply $\Ro_{2\pi/6}$ to $F(001101)$ then we obtain $F(101101)$. That is, $\Isot{\Ro_{2 \pi/6}}{F(001101)}{F(100110)}$. It is easy to see that if we apply $\Ro_{2\pi/6}$ to any flower link with $6$ components $F(a_1a_2a_3a_4a_5a_6)$, we obtain the flower link $F(a_6 a_ 1a_2 a_3 a_4 a_5)$. That is, ${\Ro_{2\pi/6}} \bigl| {F(a_1a_2a_3a_4a_5a_6)}\to$ ${F(a_6a_1a_2a_3a_4a_5)}$. 

In general, if for some positive integer $n$  we apply $\Ro_{2\pi/n}$ to a flower link with $n$ components $F(a_1a_2 \cdots a_{n-1} a_n)$, we obtain the flower link $F(a_n a_ 1a_2 \cdots a_{n-1})$. Thus $\Ro_{2\pi/n} \,\bigl|\, {F(a_1a_2\cdots a_{n-1}a_n)}$ $\to {F(a_na_1a_2\cdots a_{n-1})}$. 

Motivated by the action of $\Ro_{2\pi/n}$ on a flower link $F(a_1\cdots a_n)$, we define a mapping $\ro$ on binary words by the following rule: if $a=a_1 a_2 \ldots,a_n$ is a word, then $\ro(a)=a_2 \cdots a_n a_1$. (Here $\ro$ stands for ``rotation'', to capture the action of $\Ro_{2\pi/n}$). Thus, regardless of the value of $n$, we have that if $a$ is a word of length $n$, then $\Ro_{2\pi/n}  \,\bigl|\, {F(a)} \to F(\ro(a))$. 

If $a$ is any word of length $n$, then $\Ro_{2\pi/n}$ takes $F(a)_{i}$ to $F(\ro{a})_{i+1}$ for $i=1,\ldots,n-1$, and it takes $F(a)_{n}$
 to $F(\ro{a})_1$. Thus the $(F(a),F(\ro{a}))$-permutation under $\Ro_{2\pi/n}$ is $(2\,3\,\cdots\, n\,\,1)$. 

We shall use $\rot{}$ to denote the permutation $(2\,3\,\cdots\, n\,\,1)$ on a set $[n]$. (This is called the {\em cyclic shift permutation} in some contexts). With this notation we have that for any positive integer $n$, the $(F(a),F(\ro{a}))$-permutation under $\Ro_{2\pi/n}$ is $\sigma$. Thus we have the following expression, which brings together the isotopy $\Ro_{2\pi/n}$, the word mapping $\ro$, and the cyclic shift permutation $\sigma$, displaying the interplay between them:
\[\Isop{\Ro_{2\pi/n}}{F(a)}{F(\ro(a))}{\sigma}.
\]

Many of our arguments involve $\rot{}$ or, more generally, $\rot{}^{s}=(s{+1}\,\,\, s{+}2\,\,\,\cdots\,\, n \,\,\,1\,\,\,2\,\,\cdots \,s)$ (that is, the composition of $\rot{}$ with itself $s$ times) for some positive integer $s$. Describing the action of $\rot{}^{s}$ in detail gets a bit awkward: $\rot{}^{s}(i)=s+i$ if $s+i\le n$, and it is $n-(s+i)+2$  if $s+i > n$. This action gets a lot simpler to express using the following piece of notation, very closely related to the sum modulo $n$ of two integers $p,q$.

\vglue 0.3 cm
\noindent{\bf Notation. }{\sl If $n,p,q$ are positive integers, we define $p\opi{n} q=(p\,{+}\,q) \,\,\text{\rm mod}\,n$ if $p+q \not\equiv 0 \,\text{\rm mod}\,{n}$, and $p\opi{n} q= n$ if $p+q\equiv 0 \,\text{\rm mod}\,{n}$.}
\vglue 0.3 cm

This notation gives a cleaner way to describe the action of $\rot{}^s$ for any positive integer $s$, as $\rot{}^{s}(i)=i\opi{n} s$ for $i=1,\ldots,n$. 

With this notation, the action of $\Ro_{2\pi/n}$ thus gets captured as follows: {\em If $a$ is any word of length $n$, then $\Ro_{2\pi/n}$ takes $F(a)_{i}$ to $F(\ro{a})_{i\opi{n} 1}$ for $i=1,\ldots,n$.}

For each nonnegative integer $s$ we let $\Ro^s_{2\pi/n}$ denote the isotopy that results by iteratively applying $s$ times the isotopy $\Ro_{2\pi/n}$. Clearly, $\Ro^s_{2\pi/n} = \Ro_{s\cdot 2\pi/n}$. Therefore, if $a=a_1\cdots a_n$ then ${\Ro^s_{2\pi/n}}\,\bigl|{F(a_1 \cdots a_n)}\to$ ${F(a_{n-s+1} a_{n-s+2} \cdots a_n a_1 a_2 \cdots a_{n-s})}$. 

We let $\ro^s$ be the mapping that results by iteratively applying $\ro$ a total of $s$ times, so that $\Isot{\Ro^s_{2\pi/n}}{F(a)}{F(\ro^s(a))}$. Note that $\Ro^s_{2\pi/n}$ takes $F(a)_i$ to $F(\ro^{s}(a))_{i\opi{n}s}$ for $i=1,\ldots,n$, and that the $(F(a),F(\ro^s(a)))$-permutation under $\Ro^s_{2\pi/n}$ is $\rot{n}^s$. 

It seems worth remarking that for any positive integer $p$ the isotopy $\Ro^{p\cdot n}_{2\pi/n}=\Ro_{p\cdot n\cdot 2\pi/n}$ is a trivial isotopy that leaves every $n$-component flower link $F(a_1\cdots a_n)$ unchanged. Also note that, correspondingly, $\ro^{p\cdot n}(a)=a$, and that (if $\sigma$ is acting on $[n]$) $\rot{n}^{p\cdot n}$ is the identity permutation $\identity$.

\subsubsection{The isotopy $\Ve$ and the word mapping $\ve$}

The second and last isotopy we use in the current context of flower links is the isotopy $\Ve$ that we used in the proof of Theorem~\ref{thm:arrr}. This isotopy will have a similar effect to the one that $\Ve$ had on ring links. Indeed, as we illustrate in Figure~\ref{fig:f520}, if we apply $\Ve$ to a flower link $F(a)=F(a_1\cdots a_n)$ we obtain the flower link $F(\IO{a})=F(\IO{a_1\cdots a_n})=F(\Ov{a_n \cdots a_1})$. 

With this motivation we define a mapping $\ve$ on binary words, letting $\ve(a)=\IO{a}$ for every word $a$. (Here $\ve$ is meant to remind us of ``$\Ve$''). Thus for any word $a$ we have that $\Isot{\Ve}{F(a)}{F(\ve(a))}$. We note that, similarly as with ring links, the $(F(a),F(\ve({a})))$-permutation under $\Ve$ is the reverse permutation $\reverse$. Therefore $\Isop{\Ve}{F(a)}{F(\ve(a))}{\reverse}$.

\def\Etg{{\tf{F(\ve(001101))=F(\IO{001101})= F(010011)}}}

\begin{figure}[ht!]
\centering
\scalebox{0.8}{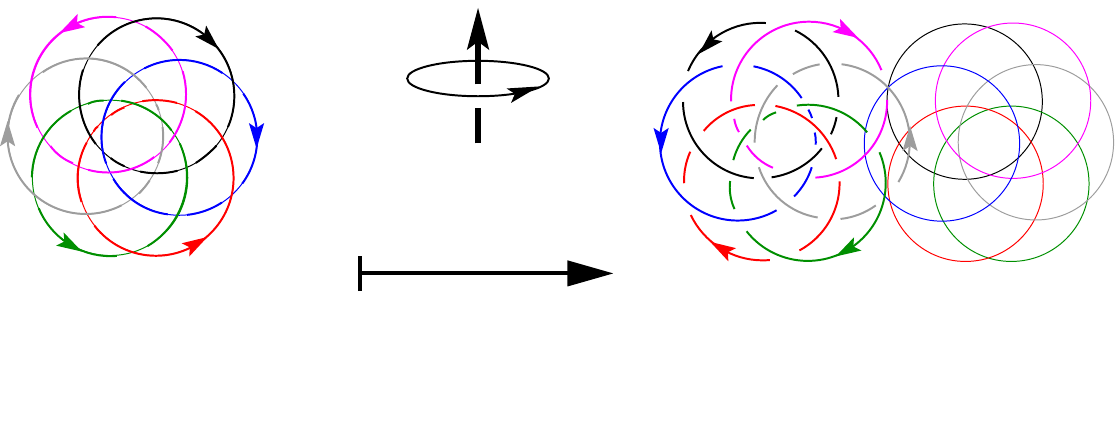}
\caption{If we apply the isotopy $\Ve$ to the flower link $F(001101)$ we obtain the flower link $F(\In{(\Ov{001101})})=F(\ve(001101))=F(010011)$. This isotopy maps the $i$-th $F(001101)_i$ component of $F(001101)$ to the $(6-i+1)$-st component $F(\ve(001101))_{6-i+1}$ of $F(\ve(001101))$. Thus the $(F(001101),F(\ve(001101))$-permutation under $\Ve$ is the reverse permutation $\reverse$ on $[6]$. Thus $\Isop{\Ve}{F(001101)}{F(\ve(001101))}{\reverse}$. In general, if we apply $\Ve$ to a flower link $F(a)$ we obtain the flower link $F(\ve(a))$, that is, $\Isot{\Ve}{F(a)}{F(\ve({a}))}$. If $F(a)$ has $n$ components then $\Ve$ maps the $i$-th component of $F(a)$ to the $(n-i+1)$-st component of $F(\ve(a))$, for $i=1,\ldots,n$. That is, $\Isop{\Ve}{F(a)}{F(\ve{(a)})}{\reverse}$.}
\label{fig:f520}
\end{figure}

\subsubsection{Combining $\Ro_{2\pi/n}$ and $\Ve$: sufficient conditions for the equivalence between two flower links $F(a)$ and $F(b)$.}

{Let $a=a_1\cdots a_n$ and $b=b_1\cdots b_n$ be words. Evidently, if $F(b)$ is obtained from $F(a)$ by applying any sequence of isotopies $\Ro_{2\pi/n}$ and $\Ve$, then $F(a)\sim F(b)$. Since applying $\Ro_{2\pi/n}$ (respectively, $\Ve$) to a link $F(a)$ yields the link $F(\ro(a))$ (respectively, $F(\ve(a))$), it follows that {\em if a word $b$ is obtained from a word $a$ by applying any sequence of $\ro$ and/or $\ve$ operations, then $F(a)\sim F(b)$}.}

\vglue 0.3 cm
\noindent{\bf Definition.} {\sl We say that a word $b$ is {\em related} to a word $a$, and write $a\equiv b$, if $b$ can be obtained from $a$ by applying a sequence of $\ro$ and $\ve$ operations.}
\vglue 0.3 cm

Using this definition, the remark at the end of the previous paragraph can be paraphrased as follows.

\begin{observation}\label{obs:ifflower}
Let $a,b$ be words. If $a\equiv b$, then $F(a)\sim F(b)$.
\end{observation}

\subsubsection{Some remarks on $\equiv$}

The relation $\equiv$ will play a crucial role in the proof of Theorem~\ref{thm:arrf}, as the goal will be to show that the sufficient condition for equivalence between flower links given in Observation~\ref{obs:ifflower} is actually necessary. Before we move on to the proof of the theorem, let us gather a few elementary facts about this relation.

{Consider $\ro$ and $\ve$ acting on words of some fixed length $n$. It is easy to see that they are the generators of a group $\grou{\ro}{\ve}$ under the composition $\circ$. Moreover, it is straightforward to verify that this group is isomorphic to the dihedral group $D_n$. Perhaps the easiest way to realize this is to note that the isotopies $\Ro_{2\pi/n}$ and $\Ve$ (the topological counterparts of $\ro$ and $\ve$, respectively) are the generators of a group naturally isomorphic to the group of symmetries of a regular $n$-gon, that is, to the dihedral group $D_n$.}

{Thus if $a,b$ are words of the same length then $a\equiv b$ if and only if there is an element $\age$ in $\grou{\ro}{\ve}$ such that $b=\age(a)$. Thus the next observation follows simply because $\grou{\ro}{\ve}$ is a group.}

\begin{remark}\label{rem:equiv}
{\em $\equiv$ is an equivalence relation}.
\end{remark}

{The dihedral nature of the group $\grou{\ro}{\ve}$ generated by $\ro$ and $\ve$ (we emphasize, throughout this discussion acting on words of some fixed length $n$) implies that it has $2n$ elements, and every element of $\grou{\ro}{\ve}$ can be written either as $\ro^s$ or as $\ro^{s}\circ\ve$ for some integer $s=0,\ldots,{n-1}$. Alternatively, every element of $\grou{\ro}{\ve}$ can be written either as $\ro^s$ or as $\ve\circ\ro^{s}$ for some integer $s=0,\ldots,{n-1}$. Thus we have the following.}

\begin{remark}\label{rem:equiv2}
Let $a,b$ be words of the same length $n$. Then $a\equiv b$ if and only if there is an $s\in\{0,\ldots,n-1\}$ such that either $b=\ro^{s}(a)$ or $b=\ro^{s}\circ\ve(a)$. In particular, the equivalence class of $a$ has size at most $2n$.
\end{remark}

It is easy to exhibit words whose equivalence class has size exactly $2n$, but we also note that there exist words whose equivalence class is much smaller: indeed, if $a$ is monotone, then its equivalence class consists only of $a$ and $\Ov{a}$.

Finally, suppose that a word $a$ can be written as a concatenation $a=D_1D_2$ of two subwords $D_1, D_2$. We say that the word $b=D_2D_1$ is a {\em shift} of $a$. Note that, evidently, $b$ is a shift of $a$ if and only if $a$ is a shift of $b$. Clearly, if we iteratively apply $\ro$ to $a$ a total of $|D_2|$ times (that is, if we apply $\ro^{|D_2|}$ to $a$) we obtain $b$. Thus the following holds.

\begin{remark}\label{rem:shift}
If $a$ is a word, and $b$ is a shift of $a$, then $b\equiv a$.
\end{remark}

\subsection{Reducing Theorem~\ref{thm:arrf} to a proposition}

Given a word $a$ of length $n$, we have thus identified a collection of words $b$ (including $a$ itself) such that $F(a)\sim F(b)$: these are the words $b$ such that $b\equiv a$. The main ingredient in the proof of Theorem~\ref{thm:arrf} is that the converse statement also holds, as long as the rank of $a$ is at least six. This statement parallels Proposition~\ref{pro:arrr} from Section~\ref{sec:proofthmarrr}.

\begin{proposition}\label{pro:arrf}
Let $a$ be a word of even rank {$r\ge 6$}. If $b$ is a word such that $F(a)\sim F(b)$, then $b\equiv {a}$.
\end{proposition}

We defer the proof of this proposition to the next section (where we actually reduce it to a lemma, similarly as we reduced Proposition~\ref{pro:arrr} to two lemmas). We note that this proposition involves only words of {\em even} rank. Evidently not every word has even rank, and so we need a wider statement that includes words whose rank is odd. Fortunately, this is an easy consequence of Proposition~\ref{pro:arrf} itself.

\begin{corollary}\label{cor:arrf}
Let $a$ be a word of rank {$r\ge 6$}. If $b$ is a word such that $F(a)\sim F(b)$, then $b\equiv {a}$.
\end{corollary}

\begin{proof}
Let $a=a_1\cdots a_n$ be a word of rank $r \ge 6$, and let $b$ be a word such that $F(a)\sim F(b)$. If $r$ is even then we are done by Proposition~\ref{pro:arrf}, and so we assume that $r$ is odd. Note that $r\ge 7$.

Let $a=A_1 A_2 \cdots A_{r-1 }A_r$ be the canonical decomposition of $a$. Since $r$ is odd it follows that the concatenation $A_rA_1$ is a monotone word. If we let $A_r':=A_rA_1$ then $A_2 \cdots A_{r-1 }A_{r'}$ is the canonical decomposition of a word $a'$. Note that since $a'$ is a shift of $a$, it follows that $a'\equiv a$.

Observation~\ref{obs:ifflower} implies then that (i) $F(a')\sim F(a)$. Since by assumption $F(a)\sim F(b)$, it follows that (ii) $F(a')\sim F(b)$. Since the rank of $a'$ is the even number $r-1\ge 6$, using Proposition~\ref{pro:arrf} (i) and (ii) we obtain that $a'\equiv a$ and $a'\equiv b$. Since $\equiv$ is an equivalence relation, it follows that $a\equiv b$.
\end{proof}

\begin{proof}[Proof of Theorem~\ref{thm:arrf} (assuming Proposition~\ref{pro:arrf})]
Let $a$ be a word of length $n$. Using standard calculations one obtains that the probability that $\rank{a}$ is less than $6$ goes to $0$ as $n\to \infty$, and also the probability that the equivalence class of $a$ has size smaller than $2n$ goes to $0$ as $n\to\infty$. By Observation~\ref{obs:ifflower} and Corollary~\ref{cor:arrf} (which holds under the assumption that Proposition~\ref{pro:arrf} holds) this implies that the probability that there are {\em exactly} $2n$ words $b$ of length $n$ (including $a$) such that $F(a) \sim F(b)$ goes to $1$ as $n\to\infty$.

The one-to-one correspondence between binary words of length $n$ and elements of $\linksor{\ff_n}$ then implies that the probability that the equivalence class of a random link in $\linksor{\ff_n}$ has size $2n$ goes to $1$ as $n \to \infty$. Since $|\linksor{\ff_n}|=2^n$, Theorem~\ref{thm:arrf} follows. 
\end{proof}

\section{Proof of Proposition~\ref{pro:arrf}}\label{sec:reduceflowerproposition}

Similarly as with ring links and boot links, sublinks of flower links play a central role in the proof of Proposition~\ref{pro:arrf}, and so we start with a brief discussion on these objects.

\subsection{Sublinks of flower links}

Let $a=a_1\ldots a_n$ be a  word, and let $i_1,\ldots,i_k$ be integers such that $1 \le i_1 < \cdots < i_k \le n$. Then $a_{i_1}\cdots a_{i_k}$ is a subword of $a$, and this subword naturally corresponds to a link $F(a)_{i_1}\cup \cdots \cup F(a)_{i_k}$ (recall that $F(a)_i$ is the $i$-th component of the flower link $F(a)$). We say that $F(a)_{i_1}\cup \cdots \cup F(a)_{i_k}$ is a {\em sublink} of $F(a)$, and for brevity we use $F(a)_{i_1,\ldots,i_k}$ to denote it.

It is worth noting that this notation is consistent with the way we denote a single component of $F(a)$: if $k=1$ then we have a single integer $i_1$, and so the corresponding sublink consists of the component $F(a)_{i_1}$.

We say that the flower link $F(a)$ is {\em oscillating} if the word $a$ is oscillating, and we say that the sublink $F(a)_{i_1,\ldots,i_k}$ of $F(a)$ is {\em oscillating} if $a_{i_1} \cdots a_{i_k}$ is an oscillating subword of $a$. Note that obviously no oscillating sublink of $F(a)$ can have size larger than the rank $r$ of $a$, since this is the length of a longest oscillating subword of $a$.

\subsection{Reducing Proposition~\ref{pro:arrf} to a lemma}

When we dealt with ring links in Section~\ref{sec:proofproparrr}, we reduced Proposition~\ref{pro:arrr} to Lemmas~\ref{lem:workhorsering} and~\ref{lem:oscrings2}. Here we proceed similarly. As we shall see, the proposition follows easily from the next two lemmas, which parallel Lemmas~\ref{lem:workhorsering} and~\ref{lem:oscrings2}, respectively.

Back in Section~\ref{sec:proofproparrr} we only stated Lemmas~\ref{lem:workhorsering} and~\ref{lem:oscrings2}, and proved them in later sections. For flower links we only need to defer the proof of Lemma~\ref{lem:oscflowers2} to later sections: as we shall see, the proof of Lemma~\ref{lem:workhorseflower} is virtually identical to the proof of Lemma~\ref{lem:workhorsering}, under the assumption that Lemma~\ref{lem:oscflowers2} holds.

\begin{lemma}\label{lem:workhorseflower}
Let $a,b$ be words with the same even rank $r\ge 6$. Suppose that $F(a)\sim F(b)$, and let $\ii$ be an $\isot{F(a)}{F(b)}$ isotopy. Let $F(a)_{i_1,\ldots,i_r}$ be an oscillating sublink of $F(a)$, and let $F(b)_{j_1,\ldots,j_r}$ be its image under $\ii$. Then $b$ is the $\pi$-image of $a$, where $\pi$ is the $(F(a_{i_1,\ldots,i_r}),F(b_{j_1,\ldots,j_r}))$-permutation under $\ii$.
\end{lemma}

\begin{lemma}\label{lem:oscflowers2}
Let $a$ be a word of even rank $r\ge 6$, and let $F(a)_{i_1,\ldots,i_r}$ be an oscillating sublink of $F(a)$ of size $r$. Let $b$ be a word such that $F(a)\sim F(b)$, and let $\ii$ be an $\isot{F(a)}{F(b)}$ isotopy. Let $F(b)_{j_1,\ldots,j_r}$ be the sublink of $F(b)$ that is the image of $F(a)_{i_1,\ldots,i_r}$ under $\ii$. Then,
\begin{enumerate}
\item[(1)] the $(F(a)_{i_1,\ldots,i_r},F(b)_{j_1,\ldots,j_r})$-permutation under $\ii$ is either $\rot{r}^s$ or $\rot{r}^{s}\circ\reverse$ for some $s\in\{0,\ldots,r-1\}$; and
\item[(2)] if $s$ is even then $b_{j_1}\cdots b_{j_r} = a_{i_1}\cdots a_{i_r}$, and if it is odd then $b_{j_1}\cdots b_{j_r} = \Ov{a_{i_1}\cdots a_{i_r}}$.
\end{enumerate}
In particular, in any case $b_{j_1} \cdots b_{j_r}$ is an oscillating subword of $b$, and so $F(b)_{j_1,,\ldots,j_r}$ is an oscillating sublink of $F(b)$.
\end{lemma}

Similarly as we proceeded back in Section~\ref{sec:proofproparrr} for ring links, we defer the proof of Lemma~\ref{lem:oscflowers2} to the next two sections, and we prove Proposition~\ref{pro:arrr} and Lemma~\ref{lem:workhorseflower} assuming Lemma~\ref{lem:oscflowers2}. We also note the following easy consequence of Lemma~\ref{lem:oscflowers2}. 

\begin{corollary}\label{cor:flowersrank}
Let $a=a_1 \cdots a_n$ be a word of even rank $r\ge 6$, and suppose that $b$ is a word such that $F(a)\sim F(b)$. Then, $\rank{b}$ is either $r$ or $r+1$. In particular, if $\rank{b}$ is even then $\rank{b}=r$.
\end{corollary}

\begin{proof}
Let $s:=\rank{b}$. Lemma~\ref{lem:oscflowers2} implies that there is an oscillating sublink of $R(b)$ of size $r$, and so it follows that $s\ge r$. In particular, $s\ge 6$. 

Suppose first that $s$ is even. Then we can apply the lemma also to an $\isot{R(b)}{R(a)}$ isotopy, obtaining that there must exist an oscillating sublink of $R(a)$ of size $s$, and so $r\ge s$. Thus in this case $r=s$.

Suppose finally that $s$ is odd, and seeking a contradiction suppose that $s > r+1$. As we argued in the proof of Corollary~\ref{cor:arrf}, there is a shift $b'$ of $b$ of rank $s-1$. Since $b'\equiv b$ and $b\equiv a$, then $b'\equiv a$, and so there is an $\isot{R(b')}{R(a)}$ isotopy $\jj$. Since $\rank{b'}=s-1$ is even, we can apply Lemma~\ref{lem:oscflowers2} to $\jj$, and obtain that there must exist an oscillating sublink of $R(a)$ of size $s-1$. But this is impossible, since $\rank{a}=r$ and $s-1>r$.
\end{proof}

\begin{proof}[Proof of Lemma~\ref{lem:workhorseflower} (assuming Lemma~\ref{lem:oscflowers2})]
This is virtually identical to the proof of Lemma~\ref{lem:workhorsering}. It suffices to replace every occurrence of ``$R$'' with ``$F$'', to invoke Corollary~\ref{cor:flowersrank} instead of Corollary~\ref{cor:ringsrank}, and to invoke Lemma~\ref{lem:oscflowers2}\,(2) instead of Lemma~\ref{lem:oscrings2}\,(2).
\end{proof}

\begin{proof}[Proof of Proposition~\ref{pro:arrf} (assuming Lemma~\ref{lem:oscflowers2})]
Let $a=a_1\cdots a_n$ be a word with even rank $r\ge 6$. Let $b=b_1\cdots b_n$ be a word such that $F(a)\sim F(b)$, and let $\ii$ be an $(F(a),F(b))$ isotopy. 

In the proof we make essential use of Lemma~\ref{lem:workhorseflower}. Since that lemma works under the assumption that {\em both} words are of the same even rank at least $6$, and we only assume that $a$ has even rank at least $6$ (as far as we know, the rank of $b$ could be any positive integer), we need to make a little adjustement in order to use that lemma, introducing a word $c$ of the same rank as $a$. 

By Corollary~\ref{cor:flowersrank} it follows that $\rank{b}$ is either $r$ or $r+1$. If $\rank{b}=r$ then we simply let $c=b$. If $\rank{b}=r+1$ then we note, as in the proof of Corollary~\ref{cor:arrf}, that there is a word $c$ such that $\rank{c}=\rank{b}-1=r$ such that $c\equiv b$. We shall prove that $c\equiv a$. Since $\equiv$ is an equivalence relation, it will follow that $b\equiv a$, as required.

Let $F(a)_{i_1,\ldots,i_r}$ be an oscillating sublink of $F(a)$, and let $F(c)_{j_1,\ldots,j_r}$ be the sublink of $F(c)$ that is the image of $F(a)_{i_1,\ldots,i_r}$ under $\ii$, and let $\pi$ be the $(F(a)_{i_1,\ldots,i_r},F(c)_{j_1,\ldots,j_r})$-permutation under $\ii$. We note that Lemma~\ref{lem:oscflowers2} implies $c_{j_1}\cdots c_{j_r}$ is oscillating. That is, $F(c)_{j_1,\ldots,j_r}$ is an oscillating sublink of $F(c)$.

Let $a=A_1 \cdots A_r=a_{i_1}^{|A_1|}\cdots a_{i_r}^{|A_r|}$ be the canonical decomposition of $a$. Since $\rank{c}=\rank{a}=r$, we let  $c=C_1 \cdots C_r=c_{j_1}^{|C_1|}\cdots c_{j_r}^{|C_r|}$ be the canonical decomposition of $c$. 

Lemma~\ref{lem:workhorseflower} implies that $c$ is the $\pi$-image of $a$, and Lemma~\ref{lem:oscflowers2}\,(1) implies that there is an $s\in\{0,\ldots,r-1\}$ such that $\pi$ is either $\rot{r}^s$ or $\rot{r}^s\circ\reverse$. 

We assume first that $\pi=\rot{r}^s$, for some $s\in\{0,\ldots,r-1\}$. Since $c$ is the $\pi$-image of $a$ it follows that $|C_{\pi(k)}|=|A_k|$ for $k=1,\ldots,r$. That is, $|C_{s\opi{r}k}|=|A_k|$ for $k=1,\ldots,r$. Therefore ($*$) $a=a_{i_1}^{|C_{s+1}|}a_{i_2}^{|C_{s+2}|} \cdots a_{i_{r-s}}^{|C_{r}|} a_{i_{r-s+1}}^{|C_1|} \cdots a_{i_{r-1}}^{|C_{s-1}|} a_{j_r}^{|C_s|}$.

Suppose that $s$ is even. Since $a_{i_1} \cdots a_{i_r}$ is oscillating and $r$ is even, this implies that $a_{i_1}\cdots a_{i_r}=a_{i_{s+1}} \cdots a_{i_{r}} a_{i_{1}} \cdots a_{i_{s}}$. Moreover, since $s$ is even then Lemma~\ref{lem:oscflowers2}\,(2) implies that $c_{j_1}\cdots c_{j_r}= a_{i_1}\cdots a_{i_r}$, and so $a_{i_{s+1}} \cdots a_{i_{r}} a_{i_{1}} \cdots a_{i_{s}}=c_{i_{s+1}} \cdots c_{i_{r}} c_{i_{1}} \cdots c_{i_{s}}$. Therefore $a_{i_1}\cdots a_{i_r}=c_{i_{s+1}} \cdots c_{i_{r}}$  $c_{i_{1}}\cdots c_{i_{s}}$, and so ($*$) implies that $a=c_{i_{s+1}}^{|C_{s+1}|} \cdots  c_{i_{r}}^{|C_{r}|} c_{i_{1}}^{|C_1|} \cdots c_{j_r}^{|C_s|}$. Thus $a$ is a shift of $c$, and so by Remark~\ref{rem:shift} $c\equiv a$.

Suppose now that $s$ is odd. Since $a_{i_1} \cdots a_{i_r}$ is oscillating and $r$ is even, this implies that $a_{i_1}\cdots a_{i_r}=\Ov{a_{i_{s+1}} \cdots a_{i_{r}} a_{i_{1}} \cdots a_{i_{s}}}$. Moreover, since $s$ is odd then Lemma~\ref{lem:oscflowers2}\,(2) implies that $c_{j_1}\cdots c_{j_r}= \Ov{a_{i_1}\cdots a_{i_r}}$, and so $\Ov{a_{i_{s+1}} \cdots a_{i_{r}} a_{i_{1}} \cdots a_{i_{s}}}=c_{i_{s+1}} \cdots c_{i_{r}} c_{i_{1}} \cdots c_{i_{s}}$. Therefore $a_{i_1}\cdots a_{i_r}=c_{i_{s+1}} \cdots c_{i_{r}} c_{i_{1}}\cdots c_{i_{s}}$, and so as in the previous case using ($*$) we obtain that $c\equiv a$.

We finally assume that $\pi=\rot{r}^s\circ\reverse$, for some $s\in\{0,\ldots,r-1\}$. Since $c$ is the $\pi$-image of $a$ it follows that $|C_{\pi(k)}|=|A_k|$ for $k=1,\ldots,r$. That is, $|C_{s\opi{r}(r-k+1)}|=|A_k|$ for $k=1,\ldots,r$. Therefore 
$$\hbox{($**$) $a=a_{i_1}^{|C_{s}|}a_{i_2}^{|C_{s-1}|} \cdots a_{i_{s-1}}^{|C_{2}|} a_{i_{s}}^{|C_1|} a_{i_{s+1}}^{|C_{r}|} \cdots a_{r-1}^{|C_{s+2}|} a_{i_r}^{|C_{s+1}|}$.}$$

Suppose that $s$ is even. Since $a_{i_1} {\cdots} a_{i_r}$ is oscillating and $r$ is even, this implies that $a_{i_1}{\cdots} a_{i_r}=\Ov{a_{i_{s}} {\cdots} a_{i_{r}} a_{i_{1}} {\cdots} a_{i_{s-1}}}$. Moreover, since $s$ is even then Lemma~\ref{lem:oscflowers2}\,(2) implies that $c_{j_1}{\cdots} c_{j_r}= a_{i_1}{\cdots} a_{i_r}$, and so $\Ov{a_{i_{s}} {\cdots} a_{i_{r}} a_{i_{1}} {\cdots} a_{i_{s-1}}}=\Ov{c_{i_{s}} {\cdots} c_{i_{r}} c_{i_{1}} {\cdots} c_{i_{s-1}}}$. Therefore $a_{i_1}{\cdots} a_{i_r}=\Ov{c_{i_{s}} {\cdots} c_{i_{r}} c_{i_{1}} {\cdots} c_{i_{s-1}}}$, and so ($**$) implies that $a=\Ov{c_{i_{s}}}^{|C_{s}|} {\cdots}  \Ov{c_{i_{1}}}^{|C_{1}|} \Ov{c_{i_{r}}}^{|C_r|} {\cdots} \Ov{c_{j_{s-1}}}^{|C_{s-1}|}$. Thus $a$ is a shift of $\Ov{c_{i_{r}}}^{|C_{r}|} {\cdots}  \Ov{c_{i_{s}}}^{|C_{s}|} \Ov{c_{i_{s}}}^{|C_{s-1}|} {\cdots}$ ${\cdots} \Ov{c_{j_{1}}}^{|C_{1}|}=\IO{c}=\ve(c)$. Remark~\ref{rem:shift} then implies that $a\equiv \ve(c)$, and since $\ve(c)\equiv c$ it follows that $c\equiv a$. 

Suppose now that $s$ is odd. Since $a_{i_1} {\cdots} a_{i_r}$ is oscillating and $r$ is even, this implies that $a_{i_1}{\cdots} a_{i_r}{=}{a_{i_{s}} {\cdots} a_{i_{r}} a_{i_{1}} {\cdots} a_{i_{s-1}}}$. Moreover, since $s$ is odd then Lemma~\ref{lem:oscflowers2}\,(2) implies that $a_{i_1}{\cdots} a_{i_r}=\Ov{c_{j_1}{\cdots} c_{j_r}} $, and so ${a_{i_{s}} {\cdots} a_{i_{r}} a_{i_{1}} {\cdots} a_{i_{s-1}}}=\Ov{c_{i_{s}} {\cdots} c_{i_{r}} c_{i_{1}} {\cdots} c_{i_{s-1}}}$. Hence $a_{i_1}{\cdots} a_{i_r}=\Ov{c_{i_{s}} {\cdots} c_{i_{r}} c_{i_{1}} {\cdots} c_{i_{s-1}}}$, and so as in the previous case using ($**$) we obtain that $c\equiv a$.
\end{proof}




\section{Towards the proof of Lemma~\ref{lem:oscflowers2}: small flower links}\label{sec:oscillatingflower}

We will prove Lemma~\ref{lem:oscflowers2} by induction on $n$, similarly as we proceeded in the proofs of {Lemmas~\ref{lem:oscrings2}} (for ring links) and~\ref{lem:oscboots2} (for boot links). Our aim in this section is to prove the lemma when $n=6$, which is the base case of the induction. As we shall see, this base case is equivalent to Claim~\ref{cla:oscflowers2} at the end of the section.

\subsection{The intrinsic symmetry groups of $F(010101)$ and $F(101010)$}

In order to establish the base case of the proof of Lemma~\ref{lem:oscflowers2} we need to calculate the intrinsic symmetry groups of $F(010101)$ and $F(101010)$. These links are hyperbolic, and so also in this case we followed the approach described in~\cite[Section 3]{canta1} to compute their intrinsic symmetry groups using {\tt SnapPy}. 

The results we obtained, which we state in Fact~\ref{fac:snappflower}, involve two particular intrinsic symmetries. Let us focus on  $F(010101)$ in this brief discussion, as the results for $F(101010)$ are totally analogous. 

The first symmetry pertains the action of the isotopy $\Ro_{2\pi/6}$ on $F(010101)$, illustrated in Figure~\ref{fig:f810}. As we already know, this isotopy takes the flower link $F(010101)$ to $F(\ro(010101))=F(101010)$. 

\begin{figure}[ht!]
\centering
\scalebox{0.8}{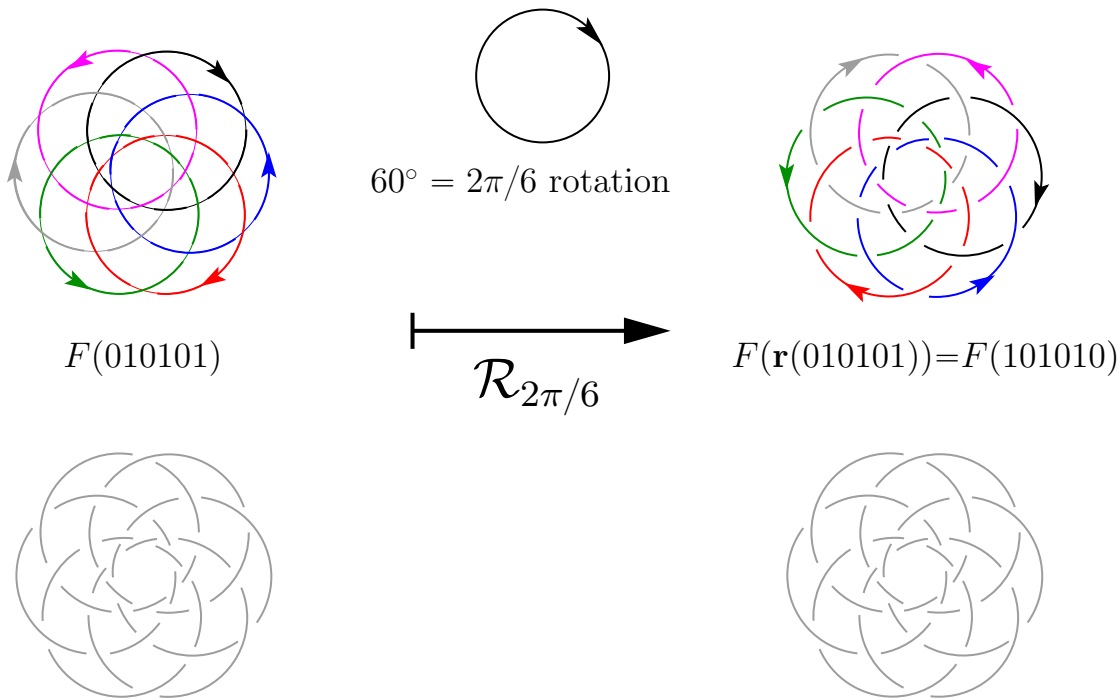}
\caption{The isotopy $\Ro_{2\pi/6}$ takes $F(010101)$ to $F(101010)$. If we ignore the orientations of the components, as we do at the bottom of this figure, we realize that this isotopy actually takes $F(010101)$ to itself. Thus it is valid to say that this isotopy takes $F(010101)$ to itself, but with the $i$-th component of $F(010101)$ taken to the $(i\opi{6}1)$-st component of $F(010101)$ with its orientation reversed. Thus this isotopy witnesses that $(1,-1,-1,-1,-1,-1,-1,\rot{6})$ is an intrinsic symmetry of $F(010101)$.}
\label{fig:f810}
\end{figure}

On the other hand, as we also illustrate in that figure, if we ignore for a moment the orientations of the components, it is valid to say that $\Ro_{2\pi/6}$ {\em takes $F(010101)$ to itself, but each component is taken to a component with its orientation reversed}. That is, for instance, $\Ro_{2\pi/6}$ takes $F(010101)_1$ to $F(010101)_2$ with its orientation reversed, that is, to $-1\cdot F(010101)_2$. Since the isotopy takes $F(010101)_i$ to $-1\cdot F(010101)_{i\opi{6}1}$ for $i=1,2,3,4,5,6$, it follows that $(1,-1,-1,-1,-1,-1,-1,(2\,3\,4\,5\,6\,1))=(1,-1,-1,-1,-1,-1,-1,\rot{6})$ is an intrinsic symmetry of $F(010101)$.

The second relevant symmetry involves the action of the isotopy $\Ve$ on $F(010101)$, illustrated in Figure~\ref{fig:f820}. As we noted in Section~\ref{sec:proofarrff}, $\Ve$ takes any flower link $F(a)$ to $F(\ve(a))$, and so in particular $\Ve$ takes the flower link $F(010101)$ to $F(\ve(010101))=F(\IO{010101})=F(\In{010101})=F(010101)$. 

\begin{figure}[ht!]
\centering
\scalebox{0.8}{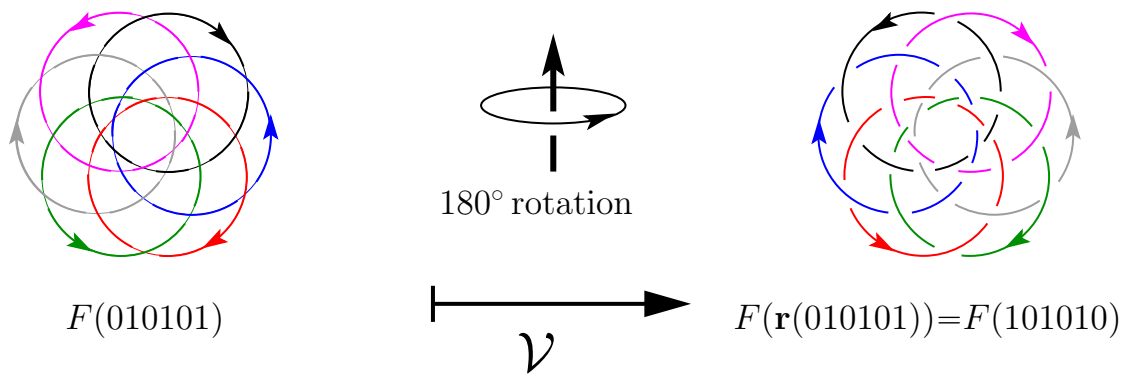}
\caption{The isotopy $\Ve$ takes $F(010101)$ to $F(010101)$ (that is, to itself), with each component taken to a component with its correct orientation. More precisely, $\Ve$ takes the $i$-th component $F(010101)_i$ of $F(010101)$ to its $(6-i+1)$-st component $F(010101)_{6-i+1}$, for $i=1,\ldots,6$. Thus $\Ve$ witnesses that $(1,1,1,1,1,1,1,\reverse)$ is an intrinsic symmetry of $F(010101)$.}
\label{fig:f820}
\end{figure}

That is, $\Ve$ {\em takes $F(010101)$ to itself, and each component is taken to a component with its correct orientation}. That is, $\Ve$ takes $F(010101)_i$ to $1\cdot F(010101)_{6-i+1}$ for $i=1,\ldots,6$. Therefore $(1,1,1,1,1,1,1,(6\,5\,4\,3\,2\,1))=(1,1,1,1,1,1,1,\reverse)$ is an intrinsic symmetry of $F(010101)$.

The next result, which we obtained using {\tt SnapPy}, attests to the relevance of these two symmetries in the intrinsic symmetry groups of $F(010101)$ and $F(101010)$.

\begin{fact}\label{fac:snappflower}
The intrinsic symmetry group of $F(010101)$ is isomorphic to the dihedral group $D_6$, and it is generated by $(1,-1,-1,-1,-1,-1,-1,\rot{6})$ and $(1,1,1,1,1,1,1,\reverse)$. The intrinsic symmetry group of $F(101010)$ is identical. 
\end{fact}

\subsection{The base case of the proof of Lemma~\ref{lem:oscflowers2}}\label{sub:smallflower}

The next statement is the main result in this section, which corresponds to the base case of the proof of Lemma~\ref{lem:oscflowers2}. As we shall see in the next section, even though this claim is stated in terms of oscillating links, and not in terms of oscillating sublinks (as Lemma~\ref{lem:oscflowers2}) this statement is indeed equivalent to the case $n=6$ of that lemma.


\begin{claim}\label{cla:oscflowers2}
Let $a=a_1a_2a_3a_4a_5a_6$ be an oscillating word of length $6$. Let $b=b_1b_2b_3b_4b_5b_6$ be a word such that $F(a)\sim F(b)$, and let $\ii$ be an $\isot{F(a)}{F(b)}$ isotopy. Then,
\begin{enumerate}
\item[(1)] the $(F(a),F(b))$-permutation under $\ii$ is either $\rot{6}^s$ or $\rot{6}^s\circ\reverse$, for some $s\in\{0,\ldots,5\}$; and 
\item[(2)] if $s$ is even then $b=a$, and if $s$ is odd then $b=\Ov{a}$.
\end{enumerate}
In particular, in any case $b$ is an oscillating word, and so $F(b)$ is an oscillating link.
\end{claim}

In the proof we use the following terminology. Suppose that $\ii$ is an isotopy that maps a link $L=L_1\cup \cdots \cup L_n$ to itself if we ignore the orientations of its components. That is, $\ii$ takes $L_i$ to $\epsilon_i\cdot L_{\pi(i)}$ for $i=1,\ldots,n$, where $\epsilon_i\in\{-1,1\}$ for $i=1,\ldots,n$. That is, $\ii$ witnesses that $L$ admits the intrinsic symmetry $(1,\epsilon_1,\ldots,\epsilon_n,\pi)$. Then we say that $(\epsilon_1,\ldots,\epsilon_n,\pi)$ is the {\em stamp of $\ii$ over $L$}. We emphasize that the stamp of an isotopy over a link exists if and only if the isotopy maps the link to itself, if we ignore the orientations of its components.

\begin{proof}
We start by pointing out that the closing sentence of the claim (that $b$ is oscillating) follows immediately from (2). However, as it happens, we need an independent verification of this fact in order to prove (1) and (2). 

This is an easy task using {\tt SageMath}: we found out that the Jones polynomials $V_{F(010101)}$ of $F(010101)$ and $V_{F(101010)}$ of $F(101010)$ are the same (this was of course expected, since these links are equivalent), and the Jones polynomial of any other flower link of size $6$ is distinct from $V_{F(010101)}$. Thus $b$ is either $a$ or $\Ov{a}$. 

Let $\ii$ be an $F(a)\mapsto F(b)$ isotopy, and let $\pi$ be the $(F(a),F(b))$-permutation under $\ii$. Since $b$ is either $a$ or $\Ov{a}$ it follows that $F(a)$ and $F(b)$ are equivalent if we ignore the orientations of their components, and so $\ii$ has a stamp $\Theta=(\epsilon_1,\ldots,\epsilon_6,\pi)$. Thus $(1,\epsilon_1,\ldots,\epsilon_6,\pi)$ is an intrinsic symmetry of $F(a)$, and so Fact~\ref{fac:snappflower} implies that $\Theta$ is either $(-1^{s},-1^{s},-1^{s},-1^{s},-1^{s},-1^{s},\rot{6}^{s})$ or $(-1^{s},-1^{s},-1^{s},-1^{s},-1^{s},-1^{s},\rot{6}^{s}\circ\reverse)$ for some $s\in \{0,\ldots,5\}$. In particular, $\pi$ is either $\rot{6}^{s}$ or $\rot{6}^s\circ\reverse$ for some $s\in\{0,\ldots,5\}$. Thus (1) holds. 

If $s$ is even then $\Theta$ is either $(1,1,1,1,1,1,\rot{6}^s)$ or $(1,1,1,1,1,1,\rot{6}^s\circ\reverse)$. In either case $\ii$ maps each component of $F(a)$ to a component of itself with its correct orientation, and so $b=a$. 

Finally, if $s$ is odd then $\Theta$ is either $(-1,-1,-1,-1,-1,-1,\rot{6}^s)$ or $(-1,-1,-1,-1,-1,-1,\rot{6}^s\circ\reverse)$. In either case $\ii$ maps each component of $F(a)$ to a component of $F(b)$ with its orientation reversed, and so $b=\Ov{a}$.
\end{proof}

\section{Proof of Lemma~\ref{lem:oscflowers2}}\label{sec:proofoscflowers2}

As in the proofs of Lemmas~\ref{lem:oscrings2} and~\ref{lem:oscboots2}, sublinks of flower links will play a central role in the proof of Lemma~\ref{lem:oscflowers2}, and so we start the discussion by laying out some basic facts about them.

\subsection{Sublinks of flower links are equivalent to flower links}

Let $a=a_1\ldots a_n$ be a  word, and let $i_1,\ldots,i_k$ be integers such that $1 \le i_1 < \cdots < i_k \le n$. Then $a_{i_1}\cdots a_{i_k}$ is a subword of $a$, and this subword naturally corresponds to a sublink $F(a)_{i_1}\cup \cdots \cup F(a)_{i_k}$ of $F(a)$. For brevity, we use $F(a)_{i_1,\ldots,i_k}$ to denote this link. 

Similarly as in the cases of ring links and boot links, as we illustrate in Figure~\ref{fig:f800} the special structure of the arrangement $\ff_n$ implies that the sublink $F(a)_{i_1,\ldots,i_k}$ of $F(a)$ is equivalent to the flower link $F(a_{i_1} \cdots a_{i_k})$, via a {strong} isotopy. Loosely speaking, one can ``bring together'' some components of $F(a)_{i_1,\ldots,i_k}$ until all the components are placed exactly in the same way as the components of $F(a_{i_1}\cdots a_{i_k})$, and we can also reverse this process to take the components of $F(a_{i_1}\cdots a_{i_k})$ to $F(a)_{i_1,\ldots,i_k}$. 

\begin{observation}\label{obs:subf}
Let $a=a_1\cdots a_n$ be a word. If $a_{i_1}\cdots a_{i_k}$ is any subword of $a$, then there exists an $\isop{F(a)_{i_1,\ldots,i_k}}{F(a_{i_1} \cdots a_{i_k})}{\identity}$ isotopy, and there exists an $\isop{F(a_{i_1}\cdots a_{i_k})}{F(a)_{i_1,\ldots,i_k}}{\identity}$ isotopy. 
\end{observation}

\def\Btg{{\tg{B(00\gl{0}1\gl{0}\gl{1})}}}
\def\tf#1{{\Scale[1.7]{#1}}}
\def\tg#1{{\Scale[1.8]{#1}}}
\def\th#1{{\Scale[3.2]{#1}}}
\def\tj#1{{\Scale[4.0]{#1}}}

\def\so#1{{\Scale[1.8]{#1}}}
\def\sp#1{{\Scale[1.2]{#1}}}
\def\sq#1{{\Scale[1.3]{#1}}}


\def\Btg{{\tg{F(01\gl{0}10\gl{1})}}}
\def\tf#1{{\Scale[1.7]{#1}}}
\def\tg#1{{\Scale[1.3]{#1}}}
\def\th#1{{\Scale[3.2]{#1}}}
\def\tj#1{{\Scale[4.0]{#1}}}

\begin{figure}[ht!]
\centering
\scalebox{0.8}{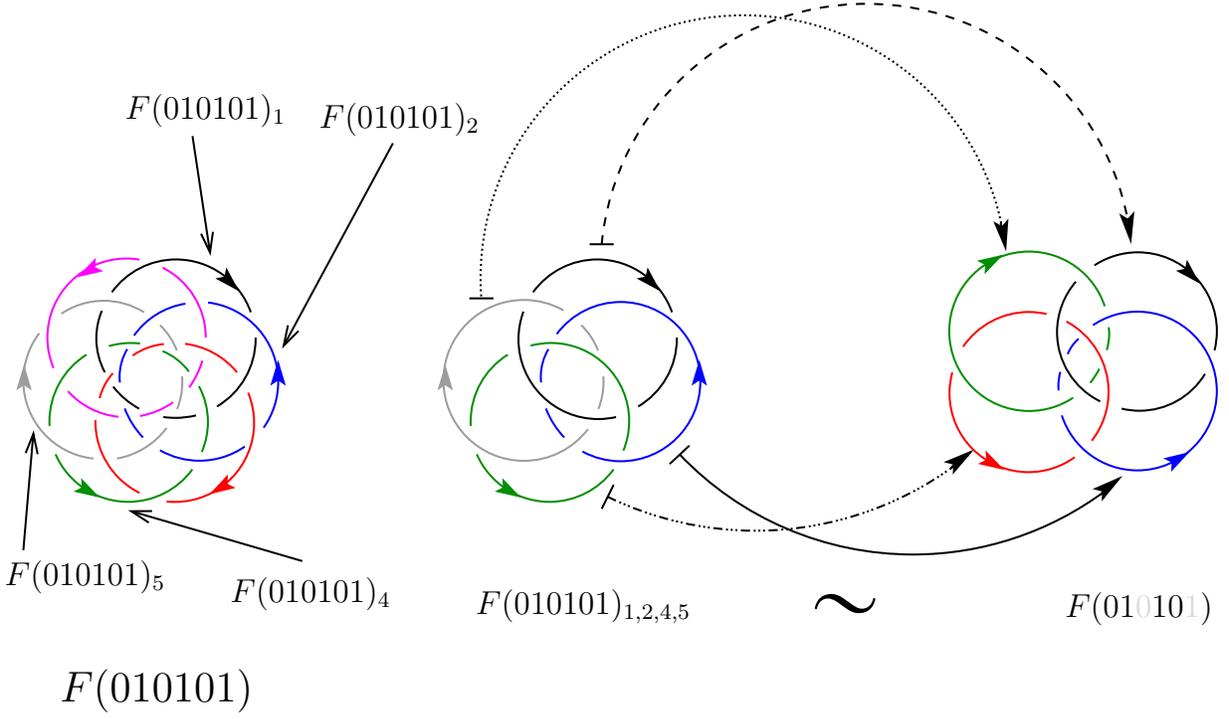}
\caption{Illustration of Observation~\ref{obs:subf}.}
\label{fig:f800}
\end{figure}

\def\so#1{{\Scale[1.8]{#1}}}
\def\sp#1{{\Scale[1.2]{#1}}}
\def\sq#1{{\Scale[1.3]{#1}}}

We thus obtain the following crucial statement, which parallels Observation~\ref{obs:subr2}.

\begin{observation}\label{obs:subf2}
Let $a=a_1\cdots a_n$ and $b=b_1\cdots b_n$ be words. Let $1\le i_1 < \cdots <i_k \le n$ and $1\le j_1 < \cdots < j_k \le n$ be integers, and let $\pi$ be a permutation of $[k]$. Then, there exists an $\isop{F(a)_{i_1,\ldots,i_k}}{F(b)_{j_1,\ldots,j_k}}{\pi}$ isotopy if and only if there exists an $\isop{F(a_{i_1}\cdots a_{i_k})}{F(b_{j_1}\cdots b_{j_k})}{\pi}$ isotopy.
\end{observation}

\subsection{Proof of Lemma~\ref{lem:oscflowers2}}

Even though in principle it is possible to prove Lemma~\ref{lem:oscflowers2} in its given form, it turns out to be easier to establish instead the following proposition, stated in terms of flower links instead of in terms of sublinks of flower links. The equivalence of this statement with Lemma~\ref{lem:oscflowers2} follows from Observation~\ref{obs:subf2}.

\begin{lemma}[Equivalent to Lemma~\ref{lem:oscflowers2}]\label{lem:flowers0}
Let $a=a_1 \cdots a_n$ be an oscillating word of even length $n\ge 6$. Let $b$ be a word such that $F(a)\sim F(b)$, and let $\ii$ be an $\isot{F(a)}{F(b)}$ isotopy. Then,
\begin{enumerate}
\item[(1)] the $(F(a),F(b))$-permutation under $\ii$ is either $\rot{r}^s$ or $\rot{r}^{s}\circ\reverse$ for some $s\in\{0,\ldots,n-1\}$; and
\item[(2)] if $s$ is even then $b=a$, and it is odd then $b=\Ov{a}$.
\end{enumerate}
In particular, in any case $b$ is oscillating.
\end{lemma}


Before we proceed to the proof we make an elementary observation on sequences. We say that a sequence $j_1,\ldots,j_k$ of $k$ distinct integers in $[n]$ is {\em $n$-consistent} (respectively, {\em $n$-anticonsistent}) if there is an integer $t\in\{1,\ldots,k\}$ such that $j_t, j_{t+1}, \ldots, j_k, j_1, j_2, \ldots, j_{t-1}$ is increasing (respectively, decreasing). We make essential use of the following trivial observation in the proof of Lemma~\ref{lem:oscrings0}.

\begin{remark}\label{rem:seqeasy}
If a sequence $j_1,\ldots,j_n$ of distinct integers in $[n]$ is $n$-consistent (respectively, $n$-anticonsistent) then the permutation $(j_1\, j_2\, \cdots \, j_n)$ is $\rot{n}^{s}$ (respectively, $\rot{n}^{s}\circ\reverse$) for some $s\in \{0,\ldots,n-1\}$.
\end{remark}

\begin{proof}[Proof of Lemma~\ref{lem:oscflowers2}]
We proceed by induction on the length $n$ of $a$. Claim~\ref{cla:oscflowers2} establishes the statement for the case $n =6$. For the inductive step we let $m\ge 6$ be an even integer, assume that the proposition holds for oscillating words of length $m$, and prove that then it holds for an oscillating word of length $m+2$. 

Thus we let $a=a_1\cdots a_m a_{m+1} a_{m+2}$ be an oscillating word, let $b=b_1 \cdots b_m b_{m+1} b_{m+2}$ be a word such that $F(a)\sim F(b)$, and let $\ii$ be an $\isot{F(a)}{F(b)}$ isotopy. Our goal is to show (1) and (2). We let $\pi$ denote the $(F(a),F(b))$-permutation under $\ii$, and so we must show that (I) $\pi$ is either $\rot{m+2}^s$ or $\rot{m+2}^s\circ\reverse$ for some $s\in\{0,\ldots,m+1\}$; and that (II) if $s$ is even then $b=a$, and if $s$ is odd then $b=\Ov{a}$.

Before we move on to the main arguments of the inductive step let us make {the following} trivial observations, {obtained} simply because $a_1 a_2 \cdots a_{m+1} a_{m+2}$ is oscillating and $m$ is even:

\begin{enumerate}
\item[$\S$] $a_1 a_2 \cdots a_m = (a_1a_2)^{m/2}$ and $\Ov{a_1 a_2 \cdots a_m} = (\Ov{a_1a_2})^{m/2}$;  and 
\item[$\S\S$] $a_3 a_4 \cdots a_{m+2} = (a_3 a_4)^{m/2}=(a_1a_2)^{m/2}$ and $\Ov{a_3 a_4 \cdots a_{m+2}} = (\Ov{a_3a_4})^{m/2}=(\Ov{a_1a_2})^{m/2}$.
\end{enumerate}

Let $F(a)_{i_1,\ldots,i_{m}}$ be any oscillating sublink of $F(a)$ of size $m$, and let $F(b)_{j_1,\ldots,j_m}$ be the sublink of $F(b)$ that is the image of $F(a)_{i_1,\ldots,i_{m}}$ under $\ii$. We note that $j_1,\ldots,j_m$ is a sequence of $m$ distinct integers in $[m+2]$, and as such it might be $(m+2)$-consistent or $(m+2)$-anticonsistent (or neither). 

By Observation~\ref{obs:subf2}, the induction hypothesis implies that there is an $s \in\{0,\ldots,m-1\}$ such that the $({F(a)_{i_1,\ldots,{i_m}}},{F(b)_{j_1,\ldots,j_m}})$-permutation under ${\ii}$ is either $\rot{m}^{s}$ or $\rot{m}^{s}\circ\reverse$ (we emphasize, with $\rot{}$ and $\reverse$ acting on $[m]$). It is easy to see that in the former case $\pi(i_1),\ldots,\pi(i_m)$ is $(m+2)$-consistent, and in the latter case it is $(m+2)$-anticonsistent. 

We now apply this discussion to two particular oscillating sublinks of $F(a)$. Since $a$ is oscillating, then $a_{1}a_{2}\cdots a_{m-1}a_{m}$ and $a_1 a_4 a_5 \cdots a_{m+1} a_{m+2}$ are both oscillating sublinks of size $m$. The remark at the end of the previous paragraph then implies that:

$\bullet$ {$\pi(1),\pi(2),\ldots,\pi(m-1),\pi(m)$ is either 
(i) $(m+2)$-consistent or  (ii) $(m+2)$-anticonsistent}; and

$\bullet$ {$\pi(1),\pi(4),\pi(5)\ldots,\pi(m+1),\pi(m+2)$ is either (iii) $(m+2)$-consistent or (iv) $(m+2)$-anticonsistent.}

It is easy to see that (i) and (iv) cannot simultaneously hold, and that (ii) and (iii) cannot simultaneously hold. So either (i) and (iii) hold, or (ii) or (iv) hold.

It is straightforward to check that if (i) and (iii) hold, the the sequence $\pi(1),\pi(2),\ldots,\pi(m),$ $\pi(m+1),$ $\pi(m+2)$ is $(m+2)$-consistent, and if (ii) and (iv) hold, then this sequence is $(m+2)$-anticonsistent. Remark~\ref{rem:seqeasy} then implies that in the former case $\pi$ is $\rot{m+2}^{s}$ for some $s\in\{0,\ldots,m+1\}$, and in the latter case $\pi$ is $\rot{m+2}^{s}\circ\reverse$ for some $s\in\{0,\ldots,m+1\}$. We have thus proved part (I) of the inductive step.

To prove (II), suppose first that $\pi=\rot{m+2}^s$ for some $s\in\{0,\ldots,m+1\}$. Then $\ii$ takes $F(a)_{1,\ldots,m}$ to $F(b)_{1,2,\ldots,s-3,s-2,s+1,s+2,\ldots,m+1,m+2}$, and the $(F(a)_{1,\ldots,m},F(b)_{1,2,\ldots,s-3,s-2,s+1,s+2,\ldots,m+1,m+2})$-permu\-tation under $\ii$ is $\rot{}^{s}$ (with $\rot{}$ acting on $[m]$). By Observation~\ref{obs:subf2}, the induction hypothesis then implies that {if $s$ is even then $b_{1} b_2 \cdots b_{s-3} b_{s-2} b_{s+1} b_{s+2}\cdots b_{m+1} b_{m+2}=a_1 \cdots a_{m}$, and if $s$ is odd then $b_{1} b_2 \cdots b_{s-3} b_{s-2} b_{s+1} b_{s+2}\cdots b_{m+1} b_{m+2}=\Ov{a_1 \cdots a_{m}}$.} 

In view of ($\S$) and ($\S\S$), we conclude that 
$$\hbox{($*$) {\em if $s$ is even then $b_{1} b_2 \cdots b_{s-2} b_{s+1} \cdots b_{m+2}$ is $(a_1a_2)^{m/2}$, and if $s$ is odd then it is $(\Ov{a_1 a_2})^{m/2}$.}}$$
We now note that $\ii$ takes $F(a)_{3,4\ldots,m+1,m+2}$ to $F(b)_{1,2,\ldots,s-1,s,s+3,s+4,\ldots,m+1,m+2}$, and similarly as in the previous case the $(F(a)_{1,\ldots,m},F(b)_{1,2,\ldots,s-1,s,s+3,s+4,\ldots,m+1,m+2})$-permu\-tation under $\ii$ is $\rot{}^{s}$ (with $\rot{}$ acting on $[m]$). By Observation~\ref{obs:subf2}, the induction hypothesis then implies that {if $s$ is even then $b_{1} b_2 \cdots b_{s-1} b_{s} b_{s+3} b_{s+4}\cdots b_{m+1} b_{m+2}=a_3 \cdots a_{m+2}$, and if $s$ is odd then $b_{1} b_2 \cdots b_{s-1} b_{s} b_{s+3} b_{s+4}\cdots b_{m+1} b_{m+2}=\Ov{a_3 \cdots a_{m+2}}$}.

Using ($\S$) and ($\S\S$), we conclude that 
$$\hbox{($**$) {\em if $s$ is even then $b_{1} b_2 \cdots b_{s} b_{s+3} \cdots b_{m+2}$ is $(a_1a_2)^{m/2}$, and if $s$ is odd then it is $(\Ov{a_1 a_2})^{m/2}$.}}$$
Combining ($*$) and ($**$) we obtain that if $s$ is even then $b_1 b_2\cdots b_{m+2}=(a_1a_2)^{(m/2)+1}$ (that is, $b=a$), and if $s$ is odd then $b_1 b_2 \cdots b_{m+2}=\Ov{a_1a_2}^{(m/2)+1}$ (that is, $b=\Ov{a}$). Thus (II) holds.

Suppose finally that $\pi=\rot{m+2}^s\reverse$ for some $s\in\{0,\ldots,m+1\}$. Then $\ii$ takes $F(a)_{1,\ldots,m}$ to $F(b)_{1,2,\ldots,s-1,s,s+3,s+4,\ldots,m+1,m+2}$, and the $(F(a)_{1,\ldots,m},F(b)_{1,2,\ldots,s-1,s,s+3,s+4,\ldots,m+1,m+2})$-permutation under $\ii$ is $\rot{}^{s}\circ\reverse$ (with $\rot{}$ and $\reverse$ acting on $[m]$). By Observation~\ref{obs:subf2}, the induction hypothesis then implies that {if $s$ is even then $b_{1} b_2 \cdots b_{s-1} b_{s} b_{s+3} b_{s+4}\cdots b_{m+1} b_{m+2}=a_1 a_{2} \cdots a_{m}$, and if $s$ is odd then $b_{1} b_2 \cdots b_{s-1} b_{s} b_{s+3} b_{s+4}\cdots b_{m+1} b_{m+2}=\Ov{a_1 a_2 \cdots a_m}$}. 

In view of ($\S$) and ($\S\S$), we conclude that 
$$\hbox{($\dag$) {\em if $s$ is even then $b_{1} b_2 \cdots b_{s} b_{s+3} \cdots b_{m+2}$ is $(a_1a_2)^{m/2}$, and if $s$ is odd then it is $(\Ov{a_1 a_2})^{m/2}$.}}$$
We now note that $\ii$ takes $F(a)_{3,4,\ldots,m+1,m+2}$ to $F(b)_{1,2,\ldots,s-3,s-2,s+1,s+2,\ldots,m+1,m+2}$, and \\
the $(F(a)_{3,4,\ldots,m+1,m+2},F(b)_{1,2,\ldots,s-3,s-2,s+1,s+2,\ldots,m+1,m+2})$-permutation under $\ii$ is $\rot{}^{s}\circ\reverse$ (with $\rot{}$ and $\reverse$ acting on $[m]$). By Observation~\ref{obs:subf2}, the induction hypothesis then implies that if $s$ is even then $b_{1} b_2 \cdots b_{s-3} b_{s-2} b_{s+1} b_{s+2}\cdots b_{m+1} b_{m+2}=a_3 a_{4} \cdots a_{m+1} a_{m+2}$, and if $s$ is odd then $b_{1} b_2 \cdots b_{s-1} b_{s} b_{s+3} b_{s+4}\cdots b_{m+1} b_{m+2}=\Ov{a_3 a_{4} \cdots a_{m+1} a_{m+2}}$. 

Using  ($\S$) and ($\S\S$), we conclude that 
$$\hbox{($\ddag$) {\em if $s$ is even then $b_{1} b_2 \cdots b_{s-2} b_{s+1} \cdots b_{m+2}$ is $(a_1a_2)^{m/2}$, and if $s$ is odd then it is $(\Ov{a_1 a_2})^{m/2}$.}}$$
Combining ($\dag$) and ($\ddag$) we obtain that if $s$ is even then $b_1 \cdots b_{m+2}=(a_1a_2)^{(m/2)+1}$ (that is, $b=a$), and if $s$ is odd then $b_1 \cdots b_{m+2}=(\Ov{a_1a_2})^{(m/2)+1}$ (that is, $b=\Ov{a}$). Therefore, (II) holds.
\end{proof}

\section{Concluding Remarks}\label{sec:concludingremarks}

Throughout this paper we worked exclusively with oriented links, and in particular our main results (namely Theorems~\ref{thm:arrr},~\ref{thm:arrs}, and~\ref{thm:arrf}) are stated in terms of oriented links. However, as we briefly mentioned in Section~\ref{sec:intro}, it is possible to derive analogous statements for unoriented links. 

Given an arrangement $\aa$, we use $\linksu{\aa}$ to denote the collection of all positive {\em unoriented} links that project to $\aa$. Similarly as for oriented links, we let $\numbu{\aa}$ denote the number of non-equivalent links in $\linksu{\aa}$ (that is, the number of equivalence classes in $\linksu{\aa}$).

We have the following consequences of Theorems~\ref{thm:arrr},~\ref{thm:arrs}, and~\ref{thm:arrf}, respectively.

\begin{corollary}[The number of positive unoriented links that project to $\rr_n$]\label{cor:uarrr}
\[
\numbu{\rr_n} = \biggl(\frac{1}{4} + o(n)\biggr) \cdot 2^n.
\]
\end{corollary}

\begin{corollary}[The number of positive unoriented links that project to $\bb_n$]\label{cor:uarrs}
\[
\numbu{\bb_n} = \biggl(\frac{1}{2}  + o(n)\biggr) \cdot 2^n.
\]
\end{corollary}

\begin{corollary}[The number of positive unoriented links that project to $\ff_n$]\label{cor:uarrf}
\[
\numbu{\ff_n} = \biggl(\frac{1}{{4n}} + o(n)\biggr) \cdot 2^n.
\]
\end{corollary}

\begin{proof}[Proof of Corollary~\ref{cor:uarrr}]
Since all links under consideration are positive, then every pair of components form a Hopf link. Therefore, if the orientation of one component is fixed then  the orientations of all the other components are determined (since each crossing must be positive). Hence an unoriented link that projects to a ring arrangement $\rr_n$ is the ``unorientation'' of exactly two oriented ring links $R(a)$ and $R(\Ov{a})$, for some word $a$ of length $n$. Now $R(a)$ and $R(\Ov{a})$ (as oriented links) are in the same equivalence class, as one can obtained from the other by using the isotopy $\Ho$ (see Section~\ref{sub:isotopiesring}).  Therefore the number of equivalence classes for unoriented ring links is the same as for oriented ring links, and so the corollary follows from Theorem \ref{thm:arrr}.
\end{proof}

\begin{proof}[Proof of Corollary~\ref{cor:uarrs}]
As in the proof of Corollary~\ref{cor:uarrr}, an unoriented link that projects to a boot arrangement $\bb_n$ is the unorientation of exactly two oriented boot links $B(a)$ and $B(\Ov{a})$, for some word $a$ of length $n$. Now if the rank of $a$ is at least $6$ then Proposition~\ref{pro:arrs} implies that $B(a)$ and $B(\Ov{a})$ (as oriented links) are {\em not} in the same equivalence class. 

Since the probability that $a$ has rank at least $6$ goes to $1$ as $n$ goes to infinity, it follows that with high probability the equivalence classes of $B(a)$ and $B(\Ov{a})$ are distinct and they merge into a single equivalence class when they are regarded as unoriented links. Therefore the corollary follows from Theorem~\ref{thm:arrs}.
\end{proof}

\begin{proof}[Proof of Corollary~\ref{cor:uarrf}]
As in the proof of the previous two corollaries, an unoriented link that projects to a flower arrangement $\ff_n$ is the unorientation of exactly two oriented flower links $F(a)$ and $F(\Ov{a})$, for some word $a$ of length $n$. Now if the rank of $a$ is at least $6$ then Proposition~\ref{pro:arrf} implies that $F(a)$ and $F(\Ov{a})$ (as oriented links) are {\em not} in the same equivalence class: indeed, it is not difficult to show that the probability that a random word $a$ of size $n$ satisfies that $a\equiv \Ov{a}$ goes to $0$ as $n$ goes to infinity. Therefore with high probability the equivalence classes of $F(a)$ and $F(\Ov{a})$ are distinct and they merge into a single equivalence class when they are regarded as unoriented links. Therefore the corollary follows from Theorem~\ref{thm:arrf}.
\end{proof}

As we mentioned in the proof of Corollary~\ref{cor:uarrf}, for every word $a$ the flower links $F(a)$ and $F(\Ov{a})$ are equivalent as unoriented links, since in $\Ov{a}$ all orientations are opposite to those in $a$, and since all crossings are positive then all the over/under assignments in these links are identical. It may happen that $F(a)$ and $F(\Ov{a})$ are also equivalent as oriented links: for instance, $F(010101)$ is equivalent to $F(101010)$, implying that $F(010101)$ is invertible. However, this is rather exceptional, as in general it is not true that $\Ov{a}\equiv a$. Indeed, an elementary calculation shows that the probability that $\Ov{a}\equiv a$ goes to $0$ as $n$ goes to infinity, and so (in view of Corollary~\ref{cor:arrf}) with high probability the oriented flower link $F(a)$ is non-invertible.

\section*{Acknowledgements}

The second and fourth authors were supported by CONACYT under Proyecto Ciencia de Frontera 191952. The third author was partially supported by INSMI--CNRS.

\bibliographystyle{abbrv}
\bibliography{refs.bib}

\begin{thebibliography}{10}

\bibitem{Colin04}
C.~Adams.
\newblock {\em The knot book. An elementary introduction to the mathematical
  theory of knots}.
\newblock Providence, RI. American Mathematical Society, 2004.

\bibitem{santino}
A.~Alba, S.~Ram\'{\i}rez, and G.~Salazar.
\newblock Regular projections of the link {$L6n1$}.
\newblock {\em J. Knot Theory Ramifications}, 32(1):Paper No. 2350006, 23,
  2023.

\bibitem{sym4010143}
M.~Berglund, J.~Cantarella, M.~P. Casey, E.~Dannenberg, W.~George, A.~Johnson,
  A.~Kelley, A.~LaPointe, M.~Mastin, J.~Parsley, J.~Rooney, and R.~Whitaker.
\newblock Intrinsic symmetry groups of links with 8 and fewer crossings.
\newblock {\em Symmetry}, 4(1):143--207, 2012.

\bibitem{canta1}
J.~Cantarella, J.~Cornish, M.~Mastin, and J.~Parsley.
\newblock The 27 possible intrinsic symmetry groups of two-component links.
\newblock {\em Symmetry}, 4(1):129--142, 2012.

\bibitem{fertilitycantarella}
J.~Cantarella, A.~Henrich, E.~Magness, O.~O'Keefe, K.~Perez, E.~Rawdon, and
  B.~Zimmer.
\newblock Knot fertility and lineage.
\newblock {\em J. Knot Theory Ramifications}, 26(13):1750093, 20, 2017.

\bibitem{SnapPy}
M.~Culler, N.~M. Dunfield, M.~Goerner, and J.~R. Weeks.
\newblock Snap{P}y, a computer program for studying the geometry and topology
  of $3$-manifolds.
\newblock Available at \url{http://snappy.computop.org} (16/10/2019).

\bibitem{evenzohar}
C.~Even-Zohar, J.~Hass, N.~Linial, and T.~Nowik.
\newblock Universal knot diagrams.
\newblock {\em J. Knot Theory Ramifications}, 28(7):1950031, 30, 2019.

\bibitem{hanaki2010}
R.~Hanaki.
\newblock Pseudo diagrams of knots, links and spatial graphs.
\newblock {\em Osaka J. Math.}, 47(3):863--883, 2010.

\bibitem{hanaki2015}
R.~Hanaki.
\newblock On scannable properties of the original knot from a knot shadow.
\newblock {\em Topology Appl.}, 194:296--305, 2015.

\bibitem{fertilityhanaki}
R.~Hanaki.
\newblock On fertility of knot shadows.
\newblock {\em J. Knot Theory Ramifications}, 29(11):2050080, 6, 2020.

\bibitem{hanaki2023}
R.~Hanaki and M.~Nakamura.
\newblock A note on closed 3-braid knot shadows.
\newblock {\em J. Knot Theory Ramifications}, 32(2):21, 2023.
\newblock Id/No 2350014.

\bibitem{henryweeks}
S.~R. Henry and J.~R. Weeks.
\newblock Symmetry groups of hyperbolic knots and links.
\newblock {\em Journal of Knot Theory and Its Ramifications}, 01(02):185--201,
  1992.

\bibitem{huhtaniyama}
Y.~Huh and K.~Taniyama.
\newblock Identifiable projections of spatial graphs.
\newblock {\em J. Knot Theory Ramifications}, 13(8):991--998, 2004.

\bibitem{itotakimura}
N.~Ito and Y.~Takimura.
\newblock {$(1,2)$} and weak {$(1,3)$} homotopies on knot projections.
\newblock {\em J. Knot Theory Ramifications}, 22(14):1350085, 14, 2013.

\bibitem{fertilityito}
T.~Ito.
\newblock A note on knot fertility.
\newblock {\em Kyushu J. Math.}, 75(2):273--276, 2021.

\bibitem{livingston2021intrinsic}
C.~Livingston.
\newblock Intrinsic symmetry groups of links.
\newblock {\em Algebr. Geom. Topol.}, 23(5):2347--2368, 2023.

\bibitem{MedinaRamirezSalazar19}
C.~Medina, J.~Ram{\'\i}rez-Alfons{\'{\i}}n, and G.~Salazar.
\newblock On the number of unknot diagrams.
\newblock {\em SIAM J. Discrete Math.}, 33:306--326, 2019.

\bibitem{pams}
C.~Medina, J.~Ram{\'\i}rez-Alfons{\'{\i}}n, and G.~Salazar.
\newblock The unavoidable arrangements of pseudocircles.
\newblock {\em Proc. Amer. Math. Soc.}, 147(7):3165--3175, 2019.

\bibitem{medina1}
C.~Medina and G.~Salazar.
\newblock The knots that lie above all shadows.
\newblock {\em Topology Appl.}, 268:106922, 13, 2019.

\bibitem{ErdosSzkeres35}
G.~S. P.~Erd\H{o}s.
\newblock A combinatorial problem in geometry.
\newblock {\em Composition Math.}, 2:463--470, 1935.

\bibitem{ptaniyama}
J.~H. Przytycki and K.~Taniyama.
\newblock Almost positive links have negative signature.
\newblock {\em J. Knot Theory Ramifications}, 19(2):187--289, 2010.

\bibitem{takimura2018}
Y.~Takimura.
\newblock Regular projections of the knot {$6_2$}.
\newblock {\em J. Knot Theory Ramifications}, 27(14):1850081, 31, 2018.

\bibitem{taniyamaknots}
K.~Taniyama.
\newblock A partial order of knots.
\newblock {\em Tokyo J. Math.}, 12(1):205--229, 1989.

\bibitem{taniyamalinks}
K.~Taniyama.
\newblock A partial order of links.
\newblock {\em Tokyo J. Math.}, 12(2):475--484, 1989.

\bibitem{sagemath}
{The Sage Developers}.
\newblock {\em {S}ageMath, the {S}age {M}athematics {S}oftware {S}ystem
  ({V}ersion 10.0)}, 2023.
\newblock {\tt https://www.sagemath.org}.

\bibitem{whitten}
W.~C. Whitten, Jr.
\newblock Symmetries of links.
\newblock {\em Trans. Amer. Math. Soc.}, 135:213--222, 1969.

\end{thebibliography}
\end{document}